\documentclass[reqno]{amsart}
\usepackage[margin=1.15in]{geometry}
\usepackage{amsmath,amssymb,mathtools,bm}
\usepackage{microtype}
\usepackage{hyperref}
\usepackage{cite}
\usepackage{enumitem}

\theoremstyle{plain}
\newtheorem{theorem}{Theorem}[section]
\newtheorem{proposition}[theorem]{Proposition}

\newtheorem{lemma}[theorem]{Lemma}

\theoremstyle{definition}

\theoremstyle{remark}
\newtheorem{remark}[theorem]{Remark}

\numberwithin{equation}{section}
\allowdisplaybreaks

\begin{document}

\author{Phuong Le}
\address{Phuong Le$^{1,2}$ (ORCID: 0000-0003-4724-7118)\newline
	$^1$Faculty of Economic Mathematics, University of Economics and Law, Ho Chi Minh City, Vietnam; \newline
	$^2$Vietnam National University, Ho Chi Minh City, Vietnam}
\email{phuongl@uel.edu.vn}


\title[Critical anisotropic $p$-Laplace equations without finite energy]{Critical anisotropic $p$-Laplace equations without finite energy: classification and a Schoen-type Harnack inequality}

\begin{abstract}
Let $N\ge 2$ and $1<p<N$. Let $H$ be a possibly non-reversible Minkowski norm on $\mathbb{R}^N$, with $H\in C^2(\mathbb{R}^N\setminus\{0\})$ and $D^2(H^2/2)$ uniformly positive definite away from the origin. We classify all nonnegative local weak solutions $u\in W^{1,p}_{\rm loc}(\mathbb{R}^N)$ of
\[
-\operatorname{div}\big(H(\nabla u)^{p-1}\nabla H(\nabla u)\big)=u^{p^*-1},\qquad p^*=\frac{Np}{N-p},
\]
without any finite-energy, integrability, or asymptotic assumption. Every nonzero solution is an anisotropic Sobolev bubble. This settles, for the full range $1<p<N$, the anisotropic analogue of Problem A of Antonini--Ciraolo--Gatti, whose Euclidean case was settled by Zhang, and removes the restriction on $p$ in the non-energy anisotropic classification of Chen--Wu--Yan--Yang. We also prove a scale-invariant Schoen-type Harnack inequality for local solutions.

The proof follows the strategy of Zhang for the Euclidean $p$-Laplacian, combining rigidity of bounded blow-up profiles and a first-contact argument with the integral bootstrap of Ciraolo and Gatti. The anisotropic analysis requires two additional ingredients. First, we establish a Bochner identity for the conjugated stress tensor in distributional form on the noncritical set, using only $C^2$ regularity of $H$. Second, we prove a two-term expansion at isolated singularities of Finsler $p$-harmonic functions, including the gradient control needed by the Pohozaev invariant. For non-reversible anisotropies the argument is formulated with the reflected norm. A Morrey-type estimate proved directly for the relevant Leray--Lions operator replaces the potential-theoretic input used in the Euclidean argument.
\end{abstract}

\maketitle

\section{Introduction}\label{sec_intro}

\subsection{The problem}

Let $N\ge 2$ and $1<p<N$. Throughout, by a \emph{norm} we mean a convex, positively one-homogeneous Minkowski norm; reversibility, namely $H(-\xi)=H(\xi)$, is not assumed. Let $H:\mathbb{R}^N\to[0,\infty)$ be such a norm and let
\[
H_0(x)=\sup_{H(\xi)\le 1}\xi\cdot x
\]
be its dual norm. We consider
\begin{equation}\label{eq_main}
-\Delta_p^H u:=-\operatorname{div}A(\nabla u)=u^{p^*-1}\quad\text{in }\mathbb{R}^N,\qquad u\ge 0,\qquad p^*=\frac{Np}{N-p},
\end{equation}
where
\[
A(\xi)=\nabla_\xi\Big(\frac{H(\xi)^p}{p}\Big)=H(\xi)^{p-1}\nabla H(\xi)\quad(\xi\neq 0),\qquad A(0)=0.
\]
Throughout the paper $H$ satisfies the following hypothesis, which is the one used by Ciraolo, Figalli and Roncoroni \cite{CFR20}.

\smallskip
\noindent\textbf{Hypothesis (H).} $H\in C^2(\mathbb{R}^N\setminus\{0\})$, and there are constants $0<\lambda\le\Lambda$ such that
\begin{equation}\label{eq_unif_convex}
\lambda\,\mathrm{Id}\le H(\xi)D^2H(\xi)+\nabla H(\xi)\otimes\nabla H(\xi)\le\Lambda\,\mathrm{Id}\qquad\text{for all }\xi\neq 0.
\end{equation}

\smallskip
We write $\hat H(\xi)=H(-\xi)$ for the reflected norm. Hypothesis (H) holds for $\hat H$ with the same constants, because $\hat HD^2\hat H+\nabla\hat H\otimes\nabla\hat H$ at $\xi$ equals $HD^2H+\nabla H\otimes\nabla H$ at $-\xi$. The dual norm of $\hat H$ is $\hat H_0(x)=\sup_{H(-\xi)\le 1}\xi\cdot x=H_0(-x)$. When $H$ is even, $\hat H=H$ and $\hat H_0=H_0$.

A \emph{local weak solution} of \eqref{eq_main} is a function $u\in W^{1,p}_{\rm loc}(\mathbb{R}^N)$ with $u\ge 0$ and
\[
\int_{\mathbb{R}^N}A(\nabla u)\cdot\nabla\varphi\,dx=\int_{\mathbb{R}^N}u^{p^*-1}\varphi\,dx\qquad\text{for all }\varphi\in C^\infty_c(\mathbb{R}^N).
\]
The right-hand side is finite because $u\in L^{p^*}_{\rm loc}(\mathbb{R}^N)$ by the Sobolev embedding. No global condition is imposed.

Set
\[
a=\frac{N-p}{p},\qquad b=\frac{N(p-1)}{p},\qquad c=a^{1-p},\qquad \alpha=\frac{N-p}{p-1},
\]
\[
p'=\frac{p}{p-1},\qquad K=\frac{1}{p'}\Big(\frac{c}{N}\Big)^{1/(p-1)}.
\]
The \emph{normalized bubble} is
\begin{equation}\label{eq_bubble}
U(x)=\big(1+K\,\hat H_0(x)^{p'}\big)^{-a}=\big(1+K\,H_0(-x)^{p'}\big)^{-a},
\end{equation}
and the \emph{bubbles} are the functions
\begin{equation}\label{eq_bubbles}
U_{z,\lambda}(x)=\lambda^{-a}\,U\Big(\frac{x-z}{\lambda}\Big),\qquad z\in\mathbb{R}^N,\ \lambda>0.
\end{equation}
They solve \eqref{eq_main} (Lemma \ref{lem_bubble}), and up to the choice of parameters they are the functions (1.6) of \cite{CFR20}.

When $H$ is the Euclidean norm, \eqref{eq_main} is the critical $p$-Laplace equation. For $p=2$ the positive solutions were classified by Caffarelli, Gidas and Spruck \cite{CGS89}. For general $p$ the classification of solutions in $D^{1,p}(\mathbb{R}^N)$ is due to Damascelli, Merch\'an, Montoro and Sciunzi \cite{DMMS14}, V\'etois \cite{Vet16} and Sciunzi \cite{Sci16}, by the method of moving planes. Since moving planes use the symmetries of the Euclidean norm, the anisotropic problem required a different approach. Ciraolo, Figalli and Roncoroni \cite{CFR20} found one, based on integral identities, and proved that every positive solution of \eqref{eq_main} in $D^{1,p}(\mathbb{R}^N)$ is a bubble. Antonini, Ciraolo and Gatti \cite{ACG26} recently gave a $P$-function proof together with decomposition and stability results, still in the finite-energy class.

Solutions which are only locally in $W^{1,p}$ are a different matter, because the integral identities of \cite{CFR20} require information at infinity. In the Euclidean case, Catino, Monticelli and Roncoroni \cite{CMR23} obtained the classification without finite energy under restrictions on $N$ and $p$, and Ou \cite{Ou25} extended the classification without further assumptions to exponents $p$ above $(N+1)/3$ (see also V\'etois \cite{Vet24}). Ciraolo and Gatti \cite{CG25} subsequently treated bounded or moderately growing solutions, while Zhang \cite{Zha26} established the unrestricted Euclidean classification, for nonlinearities $u^{p^*-1}h(u)$ with $h$ nonincreasing, for the full range $1<p<N$ together with a scale-invariant Schoen-type Harnack estimate.

For the anisotropic critical equation, Chen--Wu--Yan--Yang \cite{CWYY25} removed the finite-volume condition for $N\ge 3$ and a restricted range of $p$. In particular, their result does not cover the full interval $1<p<N$. Chen, Dai, Gui and Luo \cite{CDGL26} developed non-energy Liouville theory in convex cones for anisotropic subcritical equations, whereas their critical classification is stated for the Euclidean norm and $p>(N+1)/3$. The recent work of Antonini, Ciraolo and Gatti \cite{ACG26} treats nonnegative energy solutions for all $1<p<N$ and explicitly highlights the remaining gap between the energy and local-weak-solution theories \cite[Problem A]{ACG26}. For the Euclidean norm this gap was closed by Zhang \cite{Zha26}. The purpose of the present paper is to close it for every norm satisfying Hypothesis (H), even or not.

\subsection{Why the known arguments do not apply directly}\label{subsec_why}

Four points of the Euclidean argument depend on the Euclidean norm or on its symmetry, and each requires a new argument.

\emph{The pointwise rigidity step.} After the change of variables $w=u^{-p/(N-p)}$, rigidity comes from the sign of $\operatorname{tr}(E^2)$, where $E$ is the trace-free part of the Jacobian $DX$ of the stress field $X$ of $w$. For the anisotropic operator $DX=\mathcal{A}D^2w$, where $\mathcal{A}$ is the Hessian of $\hat H^p/p$ at $\nabla w$ (of $H^p/p$ when $H$ is even), and $E$ is not symmetric. An argument which bounds $\operatorname{tr}(E^2)$ from below by a multiple of the Frobenius norm $|E|^2$ loses a factor comparable to the condition number of $\mathcal{A}$, and this factor is not close to one when the anisotropy is large (Remark \ref{rem_frobenius}). The remedy, already used by Zhang \cite{Zha26} for the $p$-Laplacian, is the similarity $T=\mathcal{A}^{-1/2}E\mathcal{A}^{1/2}$, which is symmetric and trace-free. We observe that nothing in this device uses the Euclidean norm. The Kato inequality and the Bochner identity then hold with dimensional constants only (Lemma \ref{lem_algebra}, Proposition \ref{prop_bochner}).

\emph{Weak regularity.} The Bochner identity involves third derivatives of $w$. Under Hypothesis (H) the field $A$ is only $C^1$ away from the origin. Local $W^{1,2}$ regularity of the stress field is known from the work of Antonini, Ciraolo and Farina \cite{ACF23}, and Zhang uses such regularity in the Euclidean case. We give a self-contained argument that does not need it. The multiplier argument and the rigidity argument only use the identity on the regular set $\{\nabla w\neq 0\}$, where $w\in W^{2,2}_{\rm loc}$ and the identity holds in the sense of distributions. The proof differentiates $A$ only once: it uses continuity and ellipticity of $D A$ on compact subsets of $\mathbb R^N\setminus\{0\}$, but never differentiates $D A$. Thus $H\in C^2(\mathbb R^N\setminus\{0\})$ is enough. The critical set is handled at the end by an elementary topological argument (Step 5 in the proof of Proposition \ref{prop_profile}), rather than by a global second-order regularity theorem for the stress field.

\emph{Isolated singularities.} The first-contact argument blows down the solution at the radius where it first exceeds a multiple of the bubble. The limit has an isolated singularity at the origin, and the Pohozaev identity forces its constant term to vanish. This needs an \emph{additive} expansion $W=\kappa\hat H_0^{-\alpha}+\beta+o(1)$, together with control of the gradient of the remainder, for positive solutions of $\Delta_p^HW=0$ in a punctured Wulff ball. A quotient asymptotic $W/\hat H_0^{-\alpha}\to\kappa$, which for the Finsler $p$-Laplacian follows from the work of Hou \cite{Hou26}, leaves an error $o(\hat H_0^{-\alpha})$ that the Pohozaev identity cannot see. Theorem \ref{thm_pole} supplies the missing expansion. Its proof is self-contained and uses a flux-matching blow-up in place of the isolated-singularity theory.

\emph{Norms which are not even.} When $H$ is not even, $A$ is not odd, and the change of variables $w=u^{-p/(N-p)}$ reverses the direction of the gradient. The transformed function then solves the transformed equation for the reflected norm $\hat H$ (Lemma \ref{lem_transformed}). Accordingly the rigidity argument of Sections \ref{sec_bochner}--\ref{sec_profile} is carried out for $\hat H$, while the bubble and the fundamental solution are radial with respect to $\hat H_0(x)=H_0(-x)$. The Kilpel\"ainen--Mal\'y potential estimate \cite{KM94}, in the form used in the Euclidean argument, is formulated for operators with the real homogeneity $A(t\xi)=|t|^{p-2}tA(\xi)$ for all $t\in\mathbb{R}$, which is unavailable for a non-reversible anisotropy. The first-contact argument needs only a Morrey-type consequence of that estimate, and we prove this consequence directly for the present Leray--Lions operator (Lemma \ref{lem_morrey}).

\emph{Relation to earlier work.} The overall architecture, namely rigidity of bounded normalized profiles, a Schoen-type estimate proved by point selection and first contact, and a final integrability argument, is that of Zhang \cite{Zha26}. The change of variables $w=u^{-p/(N-p)}$ is the one used in \cite{CFR20,Ou25,CG25,Zha26}, the multiplier of Section \ref{sec_Pbound} is the pure-power case of the multiplier of \cite{Zha26}, and the integral bootstrap of Section \ref{sec_global} adapts the argument of Ciraolo and Gatti \cite{CG25}. The new ingredients are the weak Bochner identity on the noncritical set, the topological treatment of the critical set, and the two-term expansion of Theorem \ref{thm_pole}. The passage to the reflected norm follows \cite[Section 3]{CFR20}. Lemma \ref{lem_morrey} is a direct proof, without oddness, of a standard consequence of \cite{KM94}. No result of \cite{Zha26} is used as an input. Every step adapted from \cite{Zha26} is proved here in the anisotropic setting.

\subsection{Main results}

The main result is the following classification.

\begin{theorem}\label{thm_main}
Let $N\ge 2$, $1<p<N$, and let $H$ satisfy Hypothesis \textup{(H)}. If $u$ is a local weak solution of \eqref{eq_main}, then either $u\equiv 0$ or $u=U_{z,\lambda}$ for some $z\in\mathbb{R}^N$ and $\lambda>0$.
\end{theorem}

For $x_0\in\mathbb{R}^N$ and $R>0$ we write $\mathcal{W}_R(x_0)=\{x:H_0(x-x_0)<R\}$ for the Wulff ball, and $\mathcal{W}_R=\mathcal{W}_R(0)$. We also write $\hat{\mathcal{W}}_R(x_0)=\{x:\hat H_0(x-x_0)<R\}$ and $\hat{\mathcal{W}}_R=\hat{\mathcal{W}}_R(0)=-\mathcal{W}_R$. The first ingredient is a Harnack inequality of Schoen type.

\begin{theorem}\label{thm_schoen}
Let $N\ge 2$, $1<p<N$, and let $H$ satisfy Hypothesis \textup{(H)}. There is a constant $C$ depending only on $N$, $p$ and $H$ such that every nonnegative weak solution $u\in W^{1,p}_{\rm loc}(\mathcal{W}_{3R}(x_0))$ of $-\Delta_p^Hu=u^{p^*-1}$ in $\mathcal{W}_{3R}(x_0)$ satisfies
\[
\Big(\sup_{\mathcal{W}_R(x_0)}u\Big)\Big(\inf_{\mathcal{W}_{2R}(x_0)}u\Big)^{p-1}\le C\,R^{p-N}.
\]
\end{theorem}

The second ingredient is the rigidity of bounded profiles, which is where blow-up limits end up.

\begin{theorem}\label{thm_profile}
Let $N\ge 2$, $1<p<N$, and let $H$ satisfy Hypothesis \textup{(H)}. If $V\in W^{1,p}_{\rm loc}(\mathbb{R}^N)$ is a local weak solution of $-\Delta_p^HV=V^{p^*-1}$ in $\mathbb{R}^N$ with $0<V\le 1$ and $V(0)=1$, then $V=U$.
\end{theorem}

The third ingredient is an expansion at isolated singularities which we expect to be useful elsewhere. Set $\mu(x)=\hat H_0(x)^{-\alpha}=H_0(-x)^{-\alpha}$. This function is $\Delta_p^H$-harmonic in $\mathbb{R}^N\setminus\{0\}$ (Lemma \ref{lem_bubble}). For a function $W\in C^1(\hat{\mathcal{W}}_R\setminus\{0\})$ with $\Delta_p^HW=0$ weakly, the flux $\int_{\partial\hat{\mathcal{W}}_r}A(\nabla W)\cdot\nu\,dS$ does not depend on $r\in(0,R)$ (Lemma \ref{lem_flux}).

\begin{theorem}\label{thm_pole}
Let $N\ge 2$, $1<p<N$, let $H$ satisfy Hypothesis \textup{(H)}, and let $W\in C^1(\hat{\mathcal{W}}_R\setminus\{0\})$ be a positive weak solution of $\Delta_p^HW=0$ in $\hat{\mathcal{W}}_R\setminus\{0\}$. Assume that $W\le C_0\,\mu$ in $\hat{\mathcal{W}}_R\setminus\{0\}$ and that the flux $\mathcal{F}=\int_{\partial\hat{\mathcal{W}}_r}A(\nabla W)\cdot\nu\,dS$ is not zero. Then $\mathcal{F}<0$ and there are $\kappa>0$ and $\beta\in\mathbb{R}$ such that
\begin{enumerate}[label=\textup{(\roman*)}]
\item $N|\hat{\mathcal{W}}_1|(\alpha\kappa)^{p-1}=-\mathcal{F}$;
\item $W(x)-\kappa\mu(x)\to\beta$ as $x\to 0$;
\item with $z=W-\kappa\mu-\beta$ and $z_r(y)=z(ry)$,
\[
\lim_{r\to 0}\Big(\sup_{1/2\le \hat H_0(y)\le 4}|z_r(y)|+\int_{\{1\le \hat H_0\le 2\}}|\nabla z_r|^2\,dy+\sup_{\hat H_0(y)=1}r^{\alpha+1}|\nabla z(ry)|\Big)=0;
\]
\item for every $r\in(0,R)$,
\[
\int_{\partial\hat{\mathcal{W}}_r}\Big[\big(aW+x\cdot\nabla W\big)A(\nabla W)\cdot\nu-\frac{x\cdot\nu}{p}H(\nabla W)^p\Big]dS=a\,\beta\,\mathcal{F}.
\]
\end{enumerate}
\end{theorem}

\begin{remark}
The positive constant $\kappa$ in (i) and (ii) is also the limit of $W/\mu$ at the origin. For comparison, quotient asymptotics for Finsler $p$-harmonic singularities (and more general Fuchsian equations) have recently been obtained by Hou \cite{Hou26}. That result is not used here. In the proof, conservation of the flux identifies every blow-up limit. A comparison argument then bounds the remainder, and a linearized Liouville argument on annuli produces the additive constant that the Pohozaev identity requires. Only the continuity of $DA$ away from the origin is used. For the Euclidean $p$-Laplacian the behavior of solutions at isolated singularities was described by Serrin \cite{Ser65} and Kichenassamy and V\'eron \cite{KV86}. For the Finsler $p$-Laplacian with a Hardy term, Esposito, Montoro, Sciunzi and Vuono \cite[Theorem 1.3]{EMSV24} classify, for even norms, positive solutions in the whole punctured space which satisfy two-sided bounds of fundamental type. That is a global statement, while Theorem \ref{thm_pole} is local.
\end{remark}

\subsection{Organization}

Section \ref{sec_prelim} collects the properties of $H$, the regularity and comparison tools, the Pohozaev identity and a Morrey-type estimate for supersolutions, all in the form in which they are used. Section \ref{sec_bochner} introduces the transformed equation and proves the weak Bochner identity on the regular set. Section \ref{sec_Pbound} proves the sharp bound for the pressure $P$ by an explicit multiplier. Section \ref{sec_profile} proves Theorem \ref{thm_profile}. Section \ref{sec_pole} proves Theorem \ref{thm_pole}. Section \ref{sec_schoen} proves the first-contact proposition and Theorem \ref{thm_schoen}. Section \ref{sec_global} derives global integrability and finite energy and proves Theorem \ref{thm_main}.

\section{Preliminaries}\label{sec_prelim}

Throughout this section $N\ge 2$, $1<p<N$, and $H$ satisfies Hypothesis (H). Constants which depend only on $N$, $p$ and $H$ are called \emph{structural}. We write $\hat A(\xi)=\hat H(\xi)^{p-1}\nabla\hat H(\xi)$ for the vector field of $\hat H$. Since $\hat H$ also satisfies (H), every statement of this section remains true with $H$, $H_0$, $A$, $\mathcal{W}$ replaced by $\hat H$, $\hat H_0$, $\hat A$, $\hat{\mathcal{W}}$. In particular Lemma \ref{lem_norm}(e), Lemma \ref{lem_flux} and the flux form of Lemma \ref{lem_pohozaev} hold for the spheres $\partial\hat{\mathcal{W}}_r$, whatever the operator.

\subsection{The norm and the operator}

\begin{lemma}\label{lem_norm}
The following hold.
\begin{enumerate}[label=\textup{(\alph*)}]
\item There are $0<m_H\le M_H$ with $m_H|\xi|\le H(\xi)\le M_H|\xi|$ and $M_H^{-1}|x|\le H_0(x)\le m_H^{-1}|x|$. With $c_H=\max_{x\neq 0}H_0(x)/H_0(-x)\ge 1$ one has $\hat H_0\le c_HH_0$ and $H_0\le c_H\hat H_0$. For all $\xi,x$ one has $\xi\cdot x\le H(\xi)H_0(x)$. If $\xi\neq 0$ and $x\neq 0$, equality holds if and only if $x=H_0(x)\nabla H(\xi)$.
\item $H_0\in C^1(\mathbb{R}^N\setminus\{0\})$, and for $x\neq 0$, $\xi\ne 0$,
\[
H(\nabla H_0(x))=1,\qquad \nabla H(\nabla H_0(x))=\frac{x}{H_0(x)},\qquad H_0(\nabla H(\xi))=1.
\]
\item $A$ is continuous on $\mathbb{R}^N$, of class $C^1$ on $\mathbb{R}^N\setminus\{0\}$, positively homogeneous of degree $p-1$, and $A(-\xi)=-\hat A(\xi)$. For $\xi\neq 0$,
\[
A(\xi)\cdot\xi=H(\xi)^p,\qquad H_0(A(\xi))=H(\xi)^{p-1},\qquad DA(\xi)\xi=(p-1)A(\xi).
\]
Moreover $A$ is a bijection of $\mathbb{R}^N$ with inverse $A^{-1}(x)=H_0(x)^{p'-1}\nabla H_0(x)$ for $x\neq 0$.
\item For $\xi\neq 0$ the matrix $DA(\xi)=D^2(H^p/p)(\xi)$ is symmetric and, with $\lambda_p=\min\{1,p-1\}\lambda$ and $\Lambda_p=\max\{1,p-1\}\Lambda$,
\begin{equation}\label{eq_ellipticity}
\lambda_pH(\xi)^{p-2}|\eta|^2\le DA(\xi)\eta\cdot\eta\le\Lambda_pH(\xi)^{p-2}|\eta|^2\qquad(\eta\in\mathbb{R}^N).
\end{equation}
In particular $(A(\xi)-A(\eta))\cdot(\xi-\eta)>0$ for $\xi\neq\eta$, and $A(\xi)\cdot\xi\ge m_H^p|\xi|^p$, $|A(\xi)|\le C|\xi|^{p-1}$ with $C$ structural.
\item For $r>0$ the Wulff sphere $\partial\mathcal{W}_r$ is a $C^1$ hypersurface with outer unit normal $\nu=\nabla H_0/|\nabla H_0|$, and on $\partial\mathcal{W}_r$
\[
x\cdot\nu=\frac{r}{|\nabla H_0(x)|},\qquad H(\nu)=\frac{1}{|\nabla H_0(x)|},\qquad\int_{\partial\mathcal{W}_r}\frac{dS}{|\nabla H_0|}=N|\mathcal{W}_1|r^{N-1}.
\]
\item Let $0\le r_1<r_2\le\infty$ and $f\in C^1((r_1,r_2))$. If $f'\ge 0$ and $v(x)=f(H_0(x))$, or if $f'\le 0$ and $v(x)=f(\hat H_0(x))$, then, with $G=H_0$ in the first case and $G=\hat H_0$ in the second,
\[
A(\nabla v(x))=|f'(G(x))|^{p-2}f'(G(x))\,\frac{x}{G(x)}\qquad\text{on }\{r_1<G<r_2\}.
\]
If moreover $t\mapsto t^{N-1}|f'(t)|^{p-2}f'(t)$ is of class $C^1$, then $\Delta_p^Hv=t^{1-N}\big(t^{N-1}|f'|^{p-2}f'\big)'(t)$ at $t=G(x)$, pointwise and in the weak sense.
\end{enumerate}
\end{lemma}

\begin{proof}
(a) The two-sided bounds follow from the positivity and continuity of $H$ on the Euclidean unit sphere. The inequality $\xi\cdot x\le H(\xi)H_0(x)$ is the definition of $H_0$. Since $H$ is convex and positively homogeneous, it is the support function of the convex body $\{H_0\le 1\}$, and its subdifferential at $\xi\neq 0$ is the set of $y$ with $H_0(y)\le 1$ and $y\cdot\xi=H(\xi)$. As $H$ is differentiable at $\xi$, this set is $\{\nabla H(\xi)\}$. If $\xi\cdot x=H(\xi)H_0(x)$ with $x\neq 0$, the point $y=x/H_0(x)$ belongs to the subdifferential, hence $x=H_0(x)\nabla H(\xi)$. Conversely, if $x=H_0(x)\nabla H(\xi)$, Euler's identity $\nabla H(\xi)\cdot\xi=H(\xi)$ gives equality.

(b) By \eqref{eq_unif_convex}, the unit ball $\{H\le 1\}$ is strictly convex. The dual norm $H_0$ is the support function of this ball, so its subdifferential at $x\neq 0$ is the set of maximizers of $\xi\mapsto\xi\cdot x$ on $\{H\le 1\}$. Strict convexity makes this set a single point, which lies on $\{H=1\}$. A convex function whose subdifferential is a singleton at every point of an open set is differentiable there and of class $C^1$ \cite[Theorems 25.1 and 25.5]{Roc70}. Hence $H_0\in C^1(\mathbb{R}^N\setminus\{0\})$ and $H(\nabla H_0(x))=1$. Since $\nabla H_0(x)\cdot x=H_0(x)$, part (a) with $\xi=\nabla H_0(x)$ gives $x=H_0(x)\nabla H(\nabla H_0(x))$. Finally $\nabla H(\xi)$ lies in the subdifferential of $H$ at $\xi$, so $H_0(\nabla H(\xi))\le 1$, and $\nabla H(\xi)\cdot\xi=H(\xi)$ forces $H_0(\nabla H(\xi))=1$.

(c) Since $\nabla\hat H(\xi)=-\nabla H(-\xi)$, we have $\hat A(\xi)=-A(-\xi)$. The homogeneity follows from that of $H$. Continuity at $0$ follows from $|A(\xi)|\le C|\xi|^{p-1}$. The first identity is Euler's identity. The second follows from (b), since $H_0(H^{p-1}\nabla H)=H^{p-1}H_0(\nabla H)$. The third is Euler's identity for the $(p-1)$-homogeneous map $A$. For $x\neq 0$, by homogeneity and (b),
\[
A\big(H_0(x)^{p'-1}\nabla H_0(x)\big)=H_0(x)^{(p'-1)(p-1)}H(\nabla H_0(x))^{p-1}\nabla H(\nabla H_0(x))=x,
\]
because $(p'-1)(p-1)=1$. Injectivity of $A$ follows from the strict monotonicity in (d).

(d) A direct computation gives $D^2(H^p/p)=H^{p-2}\big[(p-1)\nabla H\otimes\nabla H+HD^2H\big]$. Both $\nabla H\otimes\nabla H$ and $HD^2H$ are nonnegative matrices, the latter by convexity of $H$. Therefore the bracket lies between $\min\{1,p-1\}$ and $\max\{1,p-1\}$ times $\nabla H\otimes\nabla H+HD^2H$, and \eqref{eq_unif_convex} gives \eqref{eq_ellipticity}. For $\xi\neq\eta$ the segment $\eta+t(\xi-\eta)$ meets the origin for at most one $t\in[0,1]$, so
\[
(A(\xi)-A(\eta))\cdot(\xi-\eta)=\int_0^1DA\big(\eta+t(\xi-\eta)\big)(\xi-\eta)\cdot(\xi-\eta)\,dt>0.
\]
The last two bounds follow from (a) and (c).

(e) Since $x\cdot\nabla H_0(x)=H_0(x)>0$, the gradient of $H_0$ does not vanish, and $\partial\mathcal{W}_r=\{H_0=r\}$ is a $C^1$ hypersurface. The formulas for $x\cdot\nu$ and $H(\nu)$ follow from Euler's identity and (b). The coarea formula gives $|\mathcal{W}_r|=\int_0^r\int_{\partial\mathcal{W}_s}|\nabla H_0|^{-1}dS\,ds$. As $\nabla H_0$ is homogeneous of degree zero, the inner integral equals $s^{N-1}$ times its value at $s=1$, hence is continuous in $s$. Differentiating $|\mathcal{W}_r|=r^N|\mathcal{W}_1|$ gives the last identity.

(f) By (b), $A(\nabla H_0(y))=\nabla H(\nabla H_0(y))=y/H_0(y)$ for $y\neq 0$. In the first case $\nabla v=f'(H_0)\nabla H_0$ with $f'\ge 0$, so $A(\nabla v)=(f')^{p-1}x/H_0$. In the second case $\nabla\hat H_0(x)=-\nabla H_0(-x)$ gives $\nabla v(x)=|f'(\hat H_0(x))|\,\nabla H_0(-x)$, hence $A(\nabla v(x))=|f'|^{p-1}(-x)/H_0(-x)=|f'|^{p-2}f'\,x/\hat H_0(x)$. For $\psi\in C^1$ one has $\operatorname{div}(\psi(G)x)=N\psi(G)+G\psi'(G)$, because $x\cdot\nabla G=G$. With $\psi(t)=|f'(t)|^{p-2}f'(t)/t$ this equals $t^{1-N}(t^N\psi)'$, which is the claim. The vector field $A(\nabla v)$ is then $C^1$, so the pointwise divergence is also the weak one.
\end{proof}

The next lemma turns weak divergence identities into identities for fluxes through Wulff spheres. We use it repeatedly.

\begin{lemma}\label{lem_flux}
Let $D\subset\mathbb{R}^N$ be open and let $G\in C(D;\mathbb{R}^N)$ satisfy $\operatorname{div}G=g$ in $\mathcal{D}'(D)$ for some $g\in L^1_{\rm loc}(D)$. If $0<t<s$ and $\overline{\mathcal{W}_s}\setminus\mathcal{W}_t\subset D$, then
\[
\int_{\partial\mathcal{W}_s}G\cdot\nu\,dS-\int_{\partial\mathcal{W}_t}G\cdot\nu\,dS=\int_{\mathcal{W}_s\setminus\mathcal{W}_t}g\,dx.
\]
If $\overline{\mathcal{W}_s}\subset D$, then $\int_{\partial\mathcal{W}_s}G\cdot\nu\,dS=\int_{\mathcal{W}_s}g\,dx$.
\end{lemma}

\begin{proof}
For small $\varepsilon>0$ let $\psi_\varepsilon$ be the piecewise linear function equal to $1$ on $[t,s-\varepsilon]$, to $0$ outside $(t-\varepsilon,s)$, and linear in between, and test the equation with $\psi_\varepsilon(H_0(x))$ (a Lipschitz function with compact support in $D$). This gives
\[
\frac1\varepsilon\int_{\{s-\varepsilon<H_0<s\}}G\cdot\nabla H_0\,dx-\frac1\varepsilon\int_{\{t-\varepsilon<H_0<t\}}G\cdot\nabla H_0\,dx=\int g\,\psi_\varepsilon(H_0)\,dx.
\]
By the coarea formula, the first term equals $\frac1\varepsilon\int_{s-\varepsilon}^s\Phi(\tau)\,d\tau$ with $\Phi(\tau)=\int_{\partial\mathcal{W}_\tau}G\cdot\nu\,dS$. Writing $x=\tau y$ with $y\in\partial\mathcal{W}_1$ shows $\Phi(\tau)=\tau^{N-1}\int_{\partial\mathcal{W}_1}G(\tau y)\cdot\nu(y)\,dS_y$, which is continuous in $\tau$ because $G$ is continuous. Letting $\varepsilon\to 0$ gives the first identity. The second is proved in the same way with $\psi_\varepsilon=1$ on $[0,s-\varepsilon]$.
\end{proof}

\subsection{Regularity, Harnack inequalities and comparison}

We use the following known estimates. Their structural hypotheses are those listed in Lemma \ref{lem_norm}(d), namely $A(\xi)\cdot\xi\ge m_H^p|\xi|^p$, $|A(\xi)|\le C|\xi|^{p-1}$ and \eqref{eq_ellipticity}. We write $B_r(x_0)$ for Euclidean balls.

\begin{lemma}\label{lem_harnack}
Let $\Omega\subset\mathbb{R}^N$ be open and $B_{4r}(x_0)\subset\Omega$.
\begin{enumerate}[label=\textup{(\roman*)}]
\item \textup{(Weak Harnack inequality \cite{Tru67}.)} If $u\in W^{1,p}_{\rm loc}(\Omega)$ is a nonnegative weak supersolution, $-\Delta_p^Hu\ge 0$, and $0<\gamma<N(p-1)/(N-p)$, then
\[
\Big(\frac{1}{|B_{2r}|}\int_{B_{2r}(x_0)}u^\gamma dx\Big)^{1/\gamma}\le C\inf_{B_r(x_0)}u,
\]
with $C$ depending on $N,p,H,\gamma$.
\item \textup{(Harnack inequality \cite{Ser64,Tru67}.)} If $u\in W^{1,p}_{\rm loc}(\Omega)$ is a positive weak solution of $-\Delta_p^Hu=\mathcal{V}u^{p-1}$ with $0\le\mathcal{V}\le\Lambda_0$ in $B_{4r}(x_0)$, then $\sup_{B_r(x_0)}u\le C\inf_{B_r(x_0)}u$ with $C$ depending on $N,p,H$ and $\Lambda_0r^p$.
\item \textup{($C^{1,\gamma}$ estimates \cite{DiB83,Tol84,Lie91}.)} If $u\in W^{1,p}(B_{2r}(x_0))\cap L^\infty(B_{2r}(x_0))$ is a weak solution of $-\Delta_p^Hu=f$ with $f\in L^\infty(B_{2r}(x_0))$, then $\nabla u\in C^{0,\gamma}(\overline{B_r(x_0)})$, where $\gamma\in(0,1)$ is structural and the $C^{1,\gamma}$ norm of $u$ on $B_r(x_0)$ is bounded in terms of $N,p,H,r$, $\|u\|_{L^\infty(B_{2r}(x_0))}$ and $\|f\|_{L^\infty(B_{2r}(x_0))}$.
\end{enumerate}
\end{lemma}

A consequence of (i) is the strong minimum principle: a continuous nonnegative weak supersolution in a connected open set which vanishes at one point vanishes identically, since the set of its zeros is then open and closed.

\begin{lemma}\label{lem_locreg}
Let $\Omega\subset\mathbb{R}^N$ be open and let $u\in W^{1,p}_{\rm loc}(\Omega)$, $u\ge 0$, be a weak solution of $-\Delta_p^Hu=u^{p^*-1}$ in $\Omega$. Then $u\in L^\infty_{\rm loc}(\Omega)\cap C^{1,\gamma}_{\rm loc}(\Omega)$ with $\gamma$ structural. On each connected component of $\Omega$, either $u\equiv 0$ or $u>0$.
\end{lemma}

\begin{proof}
Write $\mathcal{V}=u^{p^*-p}$, so that $-\Delta_p^Hu=\mathcal{V}u^{p-1}$ and $\mathcal{V}\in L^{N/p}_{\rm loc}(\Omega)$ because $(p^*-p)N/p=p^*$. The difficulty is that $L^{N/p}$ is the critical space for this potential, so we first improve the integrability of $u$ by the absorption argument of Brezis and Kato \cite{BK79}.

Let $B_{2\rho}(x_0)\Subset\Omega$, let $\eta\in C^\infty_c(B_{2\rho}(x_0))$ with $0\le\eta\le 1$, and let $\beta\ge 1$, $k>0$, $u_k=\min\{u,k\}$. The function $\varphi=\eta^pu\,u_k^{p(\beta-1)}$ belongs to $W^{1,p}_0(B_{2\rho}(x_0))\cap L^{p^*}$, and the weak formulation extends to it by density since $u^{p^*-1}\in L^{(p^*)'}_{\rm loc}$. We have $A(\nabla u)\cdot\nabla u=H(\nabla u)^p$, $A(\nabla u)\cdot\nabla u_k=H(\nabla u)^p\chi_{\{u<k\}}$, and $A(\nabla u)\cdot\zeta\le H(\nabla u)^{p-1}H(\zeta)$ by Lemma \ref{lem_norm}(a),(c). Hence
\[
\int\eta^pu_k^{p(\beta-1)}H(\nabla u)^p\big(1+p(\beta-1)\chi_{\{u<k\}}\big)\le\int\mathcal{V}\eta^pu^pu_k^{p(\beta-1)}+p\int\eta^{p-1}u\,u_k^{p(\beta-1)}H(\nabla u)^{p-1}H(-\nabla\eta).
\]
By Young's inequality and Lemma \ref{lem_norm}(a), the last term is at most half of the left-hand side plus $C\int|\nabla\eta|^pu^pu_k^{p(\beta-1)}$. Since $\nabla(u\,u_k^{\beta-1})=u_k^{\beta-1}\big(1+(\beta-1)\chi_{\{u<k\}}\big)\nabla u$ and $(1+(\beta-1)\chi)^p\le\beta^{p-1}(1+p(\beta-1)\chi)$ for $\chi\in\{0,1\}$, the function $v=\eta\,u\,u_k^{\beta-1}$ satisfies
\[
\int H(\nabla v)^p\le C\beta^{p-1}\Big(\int\mathcal{V}v^p+\int |\nabla\eta|^pu^pu_k^{p(\beta-1)}\Big).
\]
By the Sobolev inequality, Lemma \ref{lem_norm}(a) and H\"older's inequality,
\[
\|v\|_{L^{p^*}}^p\le C_1\beta^{p-1}\Big(\|\mathcal{V}\|_{L^{N/p}(B_{2\rho}(x_0))}\|v\|_{L^{p^*}}^p+\int |\nabla\eta|^pu^pu_k^{p(\beta-1)}\Big),
\]
with $C_1$ structural. The norm $\|v\|_{L^{p^*}}$ is finite because $u_k\le k$. If $C_1\beta^{p-1}\|\mathcal{V}\|_{L^{N/p}(B_{2\rho}(x_0))}\le 1/2$, we absorb the first term and let $k\to\infty$ to get
\begin{equation}\label{eq_bk}
\|\eta\,u^\beta\|_{L^{p^*}}^p\le 2C_1\beta^{p-1}\int |\nabla\eta|^pu^{p\beta}\,dx.
\end{equation}
Fix $s<\infty$ and a compact set $\mathcal{K}\subset\Omega$. Let $\beta_0=N/(N-p)$ and $\beta_{i+1}=\beta_iN/(N-p)$, and let $n$ be the first index with $p^*\beta_n\ge s$. Since $\int_{B_{2\rho}(x)}\mathcal{V}^{N/p}=\int_{B_{2\rho}(x)}u^{p^*}\to 0$ uniformly for $x\in\mathcal{K}$ as $\rho\to 0$, we may choose $\rho$ so small that the smallness condition holds for $\beta=\beta_n$, hence for all $\beta_i$, $i\le n$, on all balls centered in $\mathcal{K}$. Applying \eqref{eq_bk} $n+1$ times on shrinking concentric balls, starting from $u^{\beta_0}\in L^p_{\rm loc}$ (that is, $u\in L^{p^*}_{\rm loc}$), gives $u\in L^s$ near every point of $\mathcal{K}$. Hence $u\in L^s_{\rm loc}(\Omega)$ for every $s<\infty$.

Consequently $\mathcal{V}\in L^t_{\rm loc}(\Omega)$ for some $t>N/p$, and the local boundedness theorem of Serrin \cite[Theorem 1]{Ser64} gives $u\in L^\infty_{\rm loc}(\Omega)$. The right-hand side $u^{p^*-1}$ is then locally bounded, and Lemma \ref{lem_harnack}(iii) gives $u\in C^{1,\gamma}_{\rm loc}(\Omega)$. The last assertion is the strong minimum principle following Lemma \ref{lem_harnack}.
\end{proof}

\begin{lemma}[Comparison]\label{lem_comparison}
Let $D$ be a bounded open set and let $u,v\in W^{1,p}(D)\cap C(\overline D)$ satisfy $-\Delta_p^Hu\ge-\Delta_p^Hv$ weakly in $D$, that is $\int_DA(\nabla u)\cdot\nabla\varphi\ge\int_DA(\nabla v)\cdot\nabla\varphi$ for all $0\le\varphi\in W^{1,p}_0(D)$. If $u\ge v$ on $\partial D$, then $u\ge v$ in $D$.
\end{lemma}

\begin{proof}
The function $\varphi=(v-u)^+$ belongs to $W^{1,p}_0(D)$ because $v-u$ is continuous on $\overline D$ and nonpositive on $\partial D$. Testing gives $\int_D(A(\nabla v)-A(\nabla u))\cdot\nabla(v-u)^+\le 0$. By strict monotonicity (Lemma \ref{lem_norm}(d)), $\nabla(v-u)^+=0$ a.e., so $(v-u)^+=0$.
\end{proof}

\begin{lemma}[Strong comparison at a noncritical contact]\label{lem_strongcomp}
Let $D$ be a connected open set and let $u,v\in C^1(D)$ be weak solutions of $\Delta_p^Hu=\Delta_p^Hv=0$ in $D$ with $u\ge v$ in $D$. Assume that at every point $x$ of $D$ with $u(x)=v(x)$ one has $\nabla v(x)\neq 0$. Then either $u\equiv v$ or $u>v$ in $D$.
\end{lemma}

\begin{proof}
Let $Z=\{x\in D:u(x)=v(x)\}$, a relatively closed set. At $x_0\in Z$ the function $u-v\ge 0$ has a minimum, so $\nabla u(x_0)=\nabla v(x_0)\neq 0$. By continuity there is a ball $B\ni x_0$ on which all segments $[\nabla v(x),\nabla u(x)]$ lie in a compact subset of $\mathbb{R}^N\setminus\{0\}$. On $B$,
\[
A(\nabla u)-A(\nabla v)=\mathcal{B}(x)\nabla(u-v),\qquad \mathcal{B}(x)=\int_0^1DA\big(\nabla v+t(\nabla u-\nabla v)\big)\,dt,
\]
and $\mathcal{B}$ is bounded, measurable and uniformly elliptic on $B$ by \eqref{eq_ellipticity}. Thus $u-v\ge 0$ is a weak solution of the linear equation $\operatorname{div}(\mathcal{B}\nabla(u-v))=0$ vanishing at $x_0$. The strong maximum principle for linear equations in divergence form \cite[Theorem 8.19]{GT01} gives $u-v\equiv 0$ in $B$. Hence $Z$ is open, and connectedness gives the claim.
\end{proof}

\begin{remark}\label{rem_no_global_scp}
Lemma \ref{lem_strongcomp} is deliberately local. We never invoke a global nonlinear strong-comparison theorem for the anisotropic $p$-Laplacian. At a contact point with a noncritical comparator, both gradients are close to the same nonzero vector, so the segment linearization stays in a compact subset of $\mathbb R^N\setminus\{0\}$ and the problem reduces to the classical linear strong maximum principle. Hence no global strong comparison principle for $\Delta_p^H$ is needed, for any $1<p<N$.
\end{remark}

\subsection{The Pohozaev identity}

For a $C^1$ function $u$ and a continuous function $\mathcal{G}$ set
\begin{equation}\label{eq_pohozaev_field}
\mathbf{P}[u,\mathcal{G}](x)=\big(a\,u+x\cdot\nabla u\big)A(\nabla u)-x\Big(\frac{H(\nabla u)^p}{p}-\mathcal{G}(u)\Big).
\end{equation}

\begin{lemma}\label{lem_pohozaev}
Let $D$ be open, $f\in C(\mathbb{R})$, $\mathcal{G}(t)=\int_0^tf$, and let $u\in C^1(D)$ be a weak solution of $-\Delta_p^Hu=f(u)$ in $D$. Then
\[
\operatorname{div}\mathbf{P}[u,\mathcal{G}]=N\mathcal{G}(u)-a\,u\,f(u)\qquad\text{in }\mathcal{D}'(D).
\]
Consequently, by Lemma \ref{lem_flux}, for every Wulff annulus or ball as in that lemma,
\[
\int_{\partial\mathcal{W}_s}\mathbf{P}[u,\mathcal{G}]\cdot\nu\,dS-\int_{\partial\mathcal{W}_t}\mathbf{P}[u,\mathcal{G}]\cdot\nu\,dS=\int_{\mathcal{W}_s\setminus\mathcal{W}_t}\big(N\mathcal{G}(u)-a\,u\,f(u)\big)\,dx,
\]
with the inner boundary term absent for a ball.
\end{lemma}

\begin{proof}
For smooth $u$ this is a direct computation. Using $\nabla_x\big(H(\nabla u)^p/p\big)=D^2u\,A(\nabla u)$ one finds
\[
\operatorname{div}\big((x\cdot\nabla u)A(\nabla u)\big)-\operatorname{div}\Big(x\frac{H(\nabla u)^p}{p}\Big)=\Big(1-\frac Np\Big)H(\nabla u)^p+(x\cdot\nabla u)\operatorname{div}A(\nabla u),
\]
and $\operatorname{div}(a\,u\,A(\nabla u))=aH(\nabla u)^p+a\,u\operatorname{div}A(\nabla u)$, $\operatorname{div}(x\mathcal{G}(u))=N\mathcal{G}(u)+f(u)\,x\cdot\nabla u$. Since $a=(N-p)/p$ and $\operatorname{div}A(\nabla u)=-f(u)$, the terms involving $H(\nabla u)^p$ and $x\cdot\nabla u$ cancel. For $u\in C^1$ the identity is the Pucci--Serrin identity with the vector field $x$, which Degiovanni, Musesti and Squassina proved for $C^1$ (indeed locally Lipschitz) weak solutions of Euler--Lagrange equations of Lagrangians $L(s,\xi)$ of class $C^1$ and strictly convex in $\xi$ \cite[Theorem 1]{DMS03}. Here $L(s,\xi)=H(\xi)^p/p-\mathcal{G}(s)$, which satisfies these conditions by Lemma \ref{lem_norm}(d). A Pohozaev identity for weak solutions of Finsler problems is also proved by Montoro and Sciunzi \cite{MS23}. We use the Pucci--Serrin identity because it requires no regularity beyond Hypothesis~(H).
\end{proof}

\subsection{A Morrey-type estimate for supersolutions}

The following estimate is the consequence of the lower Wolff potential bound of Kilpel\"ainen and Mal\'y \cite{KM94} that the first-contact argument needs. Their framework assumes $A(t\xi)=|t|^{p-2}tA(\xi)$ for all real $t$, which fails when $H$ is not even, so we give a direct proof. It uses only the structural bounds of Lemma \ref{lem_norm}(d) and the weak Harnack inequality.

\begin{lemma}\label{lem_morrey}
Let $B_{4\rho}(x_0)\subset\Omega$ and let $u\in C(\Omega)\cap W^{1,p}_{\rm loc}(\Omega)$ be a positive weak solution of $-\Delta_p^Hu=g$ in $\Omega$ with $0\le g\in L^1_{\rm loc}(\Omega)$. Then
\[
\int_{B_\rho(x_0)}g\,dx\le C\rho^{N-p}u(x_0)^{p-1},
\]
with $C$ structural.
\end{lemma}

\begin{proof}
Let $C_A$ be a structural constant with $|A(\xi)|\le C_A|\xi|^{p-1}$, and fix $\varepsilon=\frac12\min\{p-1,\,p/(N-p)\}$. Then the two exponents $\gamma_1=p-1-\varepsilon$ and $\gamma_2=(1+\varepsilon)(p-1)$ lie in $(0,N(p-1)/(N-p))$. Let $\eta\in C^\infty_c(B_{2\rho}(x_0))$ with $0\le\eta\le 1$, $\eta=1$ on $B_\rho(x_0)$ and $|\nabla\eta|\le 2/\rho$. All integrals below are over $B_{2\rho}(x_0)$.

Testing the equation with $\eta^p$ gives
\begin{equation}\label{eq_morrey1}
\int_{B_\rho(x_0)}g\le\int g\,\eta^p=p\int\eta^{p-1}A(\nabla u)\cdot\nabla\eta\le\frac{2pC_A}{\rho}\int\eta^{p-1}|\nabla u|^{p-1}.
\end{equation}
Since $u$ is continuous and positive, $\varphi=\eta^pu^{-\varepsilon}$ is a continuous function in $W^{1,p}_0(B_{2\rho}(x_0))$. Its mollifications converge to it in $W^{1,p}$ and uniformly, with supports in a fixed compact set, so the weak formulation holds for $\varphi$ because $g\in L^1_{\rm loc}$. Thus $\varphi$ is an admissible nonnegative test function, and $\int A(\nabla u)\cdot\nabla(\eta^pu^{-\varepsilon})=\int g\eta^pu^{-\varepsilon}\ge 0$. Using $A(\nabla u)\cdot\nabla u\ge m_H^p|\nabla u|^p$ we get
\[
\varepsilon\,m_H^p\int\eta^pu^{-1-\varepsilon}|\nabla u|^p\le pC_A\int\eta^{p-1}u^{-\varepsilon}|\nabla u|^{p-1}|\nabla\eta|.
\]
Writing the integrand on the right as $\big(\eta^{p-1}u^{-(1+\varepsilon)(p-1)/p}|\nabla u|^{p-1}\big)\big(u^{(1+\varepsilon)(p-1)/p-\varepsilon}|\nabla\eta|\big)$ and applying Young's inequality with exponents $p/(p-1)$ and $p$, we obtain
\begin{equation}\label{eq_morrey2}
\int\eta^pu^{-1-\varepsilon}|\nabla u|^p\le C\int u^{\gamma_1}|\nabla\eta|^p\le C\rho^{-p}\int u^{\gamma_1}.
\end{equation}
By H\"older's inequality with the same exponents,
\[
\int\eta^{p-1}|\nabla u|^{p-1}\le\Big(\int\eta^pu^{-1-\varepsilon}|\nabla u|^p\Big)^{\frac{p-1}{p}}\Big(\int u^{\gamma_2}\Big)^{\frac1p}.
\]
By Lemma \ref{lem_harnack}(i), $\int_{B_{2\rho}(x_0)}u^{\gamma_i}\le C\rho^N\big(\inf_{B_\rho(x_0)}u\big)^{\gamma_i}\le C\rho^Nu(x_0)^{\gamma_i}$ for $i=1,2$. Combining this with \eqref{eq_morrey1} and \eqref{eq_morrey2},
\[
\int_{B_\rho(x_0)}g\le C\rho^{-1}\big(\rho^{N-p}u(x_0)^{\gamma_1}\big)^{\frac{p-1}{p}}\big(\rho^Nu(x_0)^{\gamma_2}\big)^{\frac1p}=C\rho^{N-p}u(x_0)^{p-1},
\]
because $-1+\frac{(N-p)(p-1)+N}{p}=N-p$ and $\frac{\gamma_1(p-1)+\gamma_2}{p}=p-1$.
\end{proof}

\subsection{Bubbles and the fundamental solution}

Recall $\mu(x)=\hat H_0(x)^{-\alpha}$.

\begin{lemma}\label{lem_bubble}
\begin{enumerate}[label=\textup{(\alph*)}]
\item The function $U$ in \eqref{eq_bubble} belongs to $C^1(\mathbb{R}^N)$, satisfies $A(\nabla U(x))=-\frac1N(1+K\hat H_0(x)^{p'})^{-b}x$, and solves $-\Delta_p^HU=U^{p^*-1}$ in $\mathbb{R}^N$. The same holds for every $U_{z,\lambda}$.
\item With $\kappa_U=K^{-a}$, one has $0<U\le\kappa_U\mu$ in $\mathbb{R}^N\setminus\{0\}$, $\hat H_0(x)^{\alpha}U(x)\to\kappa_U$ as $|x|\to\infty$, and
\[
M_U:=\int_{\mathbb{R}^N}U^{p^*-1}dx=|\mathcal{W}_1|K^{-b}=N|\mathcal{W}_1|(\alpha\kappa_U)^{p-1}.
\]
\item For $\kappa>0$ the function $\kappa\mu$ is a weak solution of $\Delta_p^H(\kappa\mu)=0$ in $\mathbb{R}^N\setminus\{0\}$, with
\[
A(\nabla(\kappa\mu))=-(\alpha\kappa)^{p-1}\hat H_0^{-N}x,\qquad\int_{\partial\hat{\mathcal{W}}_r}A(\nabla(\kappa\mu))\cdot\nu\,dS=-N|\mathcal{W}_1|(\alpha\kappa)^{p-1}.
\]
\end{enumerate}
\end{lemma}

\begin{proof}
(a) Let $f(t)=(1+Kt^{p'})^{-a}$. Then $f'(t)=-aKp'\,t^{p'-1}(1+Kt^{p'})^{-a-1}$. Since $(p'-1)(p-1)=1$, $(a+1)(p-1)=b$ and $(aKp')^{p-1}=a^{p-1}c/N=1/N$,
\[
|f'|^{p-2}f'(t)=-\frac1N\,t\,(1+Kt^{p'})^{-b}.
\]
The second case of Lemma \ref{lem_norm}(f) gives the formula for $A(\nabla U)$. This vector field is $C^1$ in $\mathbb{R}^N$ because $\hat H_0^{p'}\in C^1(\mathbb{R}^N)$. Moreover
\[
t^{1-N}\Big(-\frac1N\,t^N(1+Kt^{p'})^{-b}\Big)'=-(1+Kt^{p'})^{-b-1}\Big(1+Kt^{p'}-\frac{bp'}{N}Kt^{p'}\Big)=-(1+Kt^{p'})^{-b-1},
\]
because $bp'=N$. Since $(b+1)/a=p^*-1$, this is $-f^{p^*-1}$. The equation is invariant under translations and under $u\mapsto\lambda^{-a}u(\cdot/\lambda)$, because $(a+1)(p-1)+1=b+1=a(p^*-1)$.

(b) The bound follows from $1+K\hat H_0^{p'}\ge K\hat H_0^{p'}$ and $ap'=\alpha$, and the limit is immediate. By (a), and by Lemma \ref{lem_flux} and Lemma \ref{lem_norm}(e) for the spheres $\partial\hat{\mathcal{W}}_R$ (note $|\hat{\mathcal{W}}_1|=|\mathcal{W}_1|$),
\[
\int_{\hat{\mathcal{W}}_R}U^{p^*-1}=-\int_{\partial\hat{\mathcal{W}}_R}A(\nabla U)\cdot\nu\,dS=\frac{R}{N}(1+KR^{p'})^{-b}N|\mathcal{W}_1|R^{N-1}\longrightarrow|\mathcal{W}_1|K^{-b}
\]
as $R\to\infty$, since $bp'=N$. Finally $a(p-1)=b-(p-1)$ and $(\alpha K)^{p-1}=(aKp')^{p-1}=1/N$, so $N(\alpha\kappa_U)^{p-1}=N(\alpha K)^{p-1}K^{-b}=K^{-b}$.

(c) With $f(t)=\kappa t^{-\alpha}$ one has $|f'|^{p-2}f'=-(\alpha\kappa)^{p-1}t^{-(\alpha+1)(p-1)}=-(\alpha\kappa)^{p-1}t^{1-N}$, so $t^{N-1}|f'|^{p-2}f'$ is constant and the second case of Lemma \ref{lem_norm}(f) applies. The flux follows from Lemma \ref{lem_norm}(e) for $\hat H_0$.
\end{proof}

\section{The transformed equation and a weak Bochner identity}\label{sec_bochner}

\emph{Convention for Sections \ref{sec_bochner}--\ref{sec_profile}.} In these three sections $H$ denotes an arbitrary norm satisfying Hypothesis (H), and $\check H(\xi)=H(-\xi)$ is its reflection, with vector field $\check A(\xi)=-A(-\xi)$. The reason is Lemma \ref{lem_transformed} below: if $V$ solves the critical equation for $\check H$, then $w=V^{-p/(N-p)}$ solves an equation driven by $H$ itself, and all the analysis of $w$ is then written for $H$. At the end of Section \ref{sec_profile} the results are applied with $H$ replaced by $\hat H$, whose reflection is the norm of Theorem \ref{thm_main}. When $H$ is even, $\check H=H$.

Throughout these sections $\Omega\subset\mathbb{R}^N$ is open and $V\in W^{1,p}_{\rm loc}(\Omega)$ is a positive weak solution of $-\Delta_p^{\check H}V=V^{p^*-1}$ in $\Omega$. By Lemma \ref{lem_locreg} applied to $\check H$, $V\in C^{1,\gamma}_{\rm loc}(\Omega)$. We set
\begin{equation}\label{eq_transform}
w=V^{-1/a}=V^{-p/(N-p)},\qquad X=A(\nabla w),\qquad h=H(\nabla w)^p,\qquad P=\frac{b\,h+c}{w},
\end{equation}
and
\[
\Omega_{\rm reg}=\{x\in\Omega:\nabla w(x)\neq 0\}=\{x\in\Omega:\nabla V(x)\neq 0\}.
\]
The functions $w$, $X$, $h$ and $P$ are continuous on $\Omega$, and $\Omega_{\rm reg}$ is open. On $\Omega_{\rm reg}$ we write $\mathcal{A}=DA(\nabla w)$, a continuous field of symmetric positive definite matrices.

\begin{lemma}\label{lem_transformed}
In the sense of distributions, $\operatorname{div}X=P$ in $\Omega$.
\end{lemma}

\begin{proof}
Since $V=w^{-a}$, we have $\nabla V=-a\,w^{-a-1}\nabla w$. As $\check A(\xi)=-A(-\xi)$ and $A$ is positively homogeneous of degree $p-1$,
\[
\check A(\nabla V)=-A\big(a\,w^{-a-1}\nabla w\big)=-a^{p-1}w^{-(a+1)(p-1)}X=-c^{-1}w^{-b}X.
\]
For $\varphi\in C^1_c(\Omega)$ the function $c\,w^b\varphi$ is an admissible test function: on its compact support, $V$ is bounded away from zero, so $w$ and $\nabla w$ are bounded, and the weak formulation extends from $C_c^\infty$ to this $W^{1,p}_0$ test function by density. Moreover $V^{p^*-1}=w^{-a(p^*-1)}=w^{-b-1}$. Testing the equation for $V$ with it gives
\[
-\int w^{-b}X\cdot\big(w^b\nabla\varphi+b\,w^{b-1}\varphi\nabla w\big)=c\int w^{-1}\varphi.
\]
Since $X\cdot\nabla w=h$, this is $-\int X\cdot\nabla\varphi=\int(bh+c)w^{-1}\varphi$, which is the claim.
\end{proof}

\begin{lemma}\label{lem_reg}
We have $w\in W^{2,2}_{\rm loc}(\Omega_{\rm reg})$, $X\in W^{1,2}_{\rm loc}(\Omega_{\rm reg})$ and $P\in W^{1,2}_{\rm loc}(\Omega_{\rm reg})$. Almost everywhere in $\Omega_{\rm reg}$,
\begin{equation}\label{eq_chain}
DX=\mathcal{A}D^2w,\qquad \nabla h=p\,D^2w\,X,\qquad \operatorname{tr}(DX)=P,\qquad \mathcal{A}\nabla w=(p-1)X.
\end{equation}
\end{lemma}

\begin{proof}
Let $D\Subset\Omega_{\rm reg}$ be open. The set $\nabla w(\overline D)$ is a compact subset of $\mathbb{R}^N\setminus\{0\}$, so there is $\delta>0$ such that its closed $\delta$-neighborhood $\mathcal{K}$ does not contain the origin. On $\mathcal{K}$ the map $A$ is $C^1$ and $DA$ is uniformly elliptic and bounded by \eqref{eq_ellipticity}.

\emph{Step 1: $w\in W^{2,2}_{\rm loc}(D)$.} Let $e$ be a unit vector and $\Delta_\tau g(x)=(g(x+\tau e)-g(x))/\tau$. Fix $D'\Subset D$. By uniform continuity of $\nabla w$, for $|\tau|$ small the segment joining $\nabla w(x)$ and $\nabla w(x+\tau e)$ lies in $\mathcal{K}$ for $x\in D'$, so
\[
\Delta_\tau X=\mathcal{B}_\tau\nabla\Delta_\tau w,\qquad\mathcal{B}_\tau(x)=\int_0^1DA\big(\nabla w(x)+t(\nabla w(x+\tau e)-\nabla w(x))\big)dt,
\]
with $\mathcal{B}_\tau$ uniformly elliptic and bounded on $D'$, uniformly in $\tau$. Testing Lemma \ref{lem_transformed} with $\Delta_{-\tau}\varphi$, initially for $\varphi\in C_c^\infty(D')$, gives
\[
\int_{D'}\mathcal{B}_\tau\nabla\Delta_\tau w\cdot\nabla\varphi\,dx=\int_{D'}P\,\Delta_{-\tau}\varphi\,dx.
\]
For fixed sufficiently small $\tau$, both sides are continuous with respect to $\varphi$ in $W^{1,2}_0(D')$: the left-hand side because $\mathcal{B}_\tau$ is bounded and $\nabla\Delta_\tau w\in L^2(D')$, and the right-hand side because $P\in L^\infty(D')$ and the difference-quotient map is bounded from $W^{1,2}_0(D')$ to $L^2(D')$. Hence the identity extends by density to $W^{1,2}_0(D')$. We may therefore choose $\varphi=\eta^2\Delta_\tau w$, with $\eta\in C^\infty_c(D')$; choose $D''$ with $\operatorname{supp}\eta\Subset D''\Subset D'$ so that all translated supports remain in $D'$ for $|\tau|$ small. Since $P$ is continuous, $\|P\|_{L^\infty(D')}<\infty$. If $\lambda_D,\Lambda_D>0$ are ellipticity bounds for $\mathcal{B}_\tau$ on $D'$, then
\[
\begin{aligned}
\lambda_D\int \eta^2|\nabla\Delta_\tau w|^2
&\le 2\Lambda_D\int \eta|\nabla\eta|\,|\Delta_\tau w|\,|\nabla\Delta_\tau w|
   +\Big|\int P\,\Delta_{-\tau}(\eta^2\Delta_\tau w)\Big|.
\end{aligned}
\]
For compactly supported $g\in W^{1,1}$, the difference-quotient estimate $\|\Delta_{-\tau}g\|_{L^1}\le\|\partial_e g\|_{L^1}$ gives
\[
\Big|\int P\,\Delta_{-\tau}(\eta^2\Delta_\tau w)\Big|
\le \|P\|_{L^\infty(D')}
\left(2\int \eta|\nabla\eta|\,|\Delta_\tau w|+\int\eta^2|\nabla\Delta_\tau w|\right).
\]
Young's inequality therefore yields
\[
\int\eta^2|\nabla\Delta_\tau w|^2
\le C_D\left(1+\int_{D''}|\Delta_\tau w|^2\right)
\le C_D\left(1+\int_{D'}|\nabla w|^2\right),
\]
where $C_D$ is independent of $\tau$; in the last step we used the standard $L^2$ difference-quotient bound. Applying this estimate to each coordinate direction and using \cite[Lemma 7.24]{GT01} gives $\nabla w\in W^{1,2}_{\rm loc}(D)$.

\emph{Step 2.} Since $\nabla w\in W^{1,2}_{\rm loc}(D)\cap C(D)$ takes values in $\mathcal{K}$, where $A$ and $H^p$ are $C^1$, the chain rule for Sobolev functions gives $X\in W^{1,2}_{\rm loc}(D)$ with $DX=\mathcal{A}D^2w$, and $h\in W^{1,2}_{\rm loc}(D)$ with $\nabla h=pD^2w\,A(\nabla w)=pD^2w\,X$, because $\nabla_\xi(H^p)=pA$. Then $P\in W^{1,2}_{\rm loc}(D)$. By Lemma \ref{lem_transformed}, $\operatorname{tr}(DX)=\operatorname{div}X=P$ almost everywhere. The last identity in \eqref{eq_chain} is Lemma \ref{lem_norm}(c).
\end{proof}

On $\Omega_{\rm reg}$ we define, almost everywhere,
\[
E=DX-\frac PN\,\mathrm{Id},\qquad T=\mathcal{A}^{-1/2}E\,\mathcal{A}^{1/2}.
\]

\begin{lemma}\label{lem_algebra}
Almost everywhere in $\Omega_{\rm reg}$ the following hold.
\begin{enumerate}[label=\textup{(\roman*)}]
\item $T=\mathcal{A}^{1/2}D^2w\,\mathcal{A}^{1/2}-\frac PN\mathrm{Id}$. Thus $T$ is symmetric and trace-free, and $\operatorname{tr}(E^2)=|T|^2$.
\item $w\nabla P=N(p-1)D^2w\,X-P\nabla w$, equivalently $N E^{\mathsf T}\nabla w=w\nabla P$, and
\begin{equation}\label{eq_EX}
EX=\frac{w}{N(p-1)}\,\mathcal{A}\nabla P.
\end{equation}
\item \textup{(Kato inequality)} $\displaystyle\operatorname{tr}(E^2)\ge\frac{w^2\,(X\cdot\nabla P)^2}{N(N-1)\,h^2}$.
\end{enumerate}
\end{lemma}

\begin{proof}
(i) By \eqref{eq_chain}, $E=\mathcal{A}D^2w-\frac PN\mathrm{Id}$, and conjugation gives the formula for $T$. Similar matrices have the same trace, and $\operatorname{tr}E=\operatorname{tr}(DX)-P=0$. Moreover $\operatorname{tr}(E^2)=\operatorname{tr}(T^2)=|T|^2$ because $T$ is symmetric.

(ii) Differentiating $wP=bh+c$ with \eqref{eq_chain} gives $P\nabla w+w\nabla P=b\,p\,D^2w\,X$, and $bp=N(p-1)$. Next $E^{\mathsf T}=D^2w\,\mathcal{A}-\frac PN\mathrm{Id}$, so
\[
NE^{\mathsf T}\nabla w=N\,D^2w\,\mathcal{A}\nabla w-P\nabla w=N(p-1)D^2wX-P\nabla w=w\nabla P.
\]
Finally $\mathcal{A}E^{\mathsf T}=\mathcal{A}D^2w\mathcal{A}-\frac PN\mathcal{A}=E\mathcal{A}$, hence $EX=\frac{1}{p-1}E\mathcal{A}\nabla w=\frac{1}{p-1}\mathcal{A}E^{\mathsf T}\nabla w=\frac{w}{N(p-1)}\mathcal{A}\nabla P$.

(iii) Let $\tau_{\max}$ and $\tau_{\min}$ be the largest and smallest eigenvalues of $T$. The other eigenvalues sum to $-\tau_{\max}$, so by the Cauchy--Schwarz inequality their squares sum to at least $\tau_{\max}^2/(N-1)$. Hence $|T|^2\ge\frac{N}{N-1}\tau_{\max}^2$, and similarly for $\tau_{\min}$. For every $y\neq 0$ this gives
\[
|T|^2\ge\frac{N}{N-1}\Big(\frac{y\cdot Ty}{|y|^2}\Big)^2.
\]
Take $y=\mathcal{A}^{-1/2}X$. By \eqref{eq_EX}, $y\cdot Ty=X\cdot\mathcal{A}^{-1}EX=\frac{w}{N(p-1)}X\cdot\nabla P$. By \eqref{eq_chain}, $|y|^2=X\cdot\mathcal{A}^{-1}X=\frac{1}{p-1}X\cdot\nabla w=\frac{h}{p-1}$. Substituting gives (iii).
\end{proof}

\begin{remark}\label{rem_frobenius}
The Kato inequality (iii) contains no constant depending on $H$. This is the point of the conjugation. If one works instead with the Frobenius norm $|E|$ of the nonsymmetric matrix $E$, one has only $|E|^2=|\mathcal{A}^{1/2}T\mathcal{A}^{-1/2}|^2\le k(\mathcal{A})|T|^2=k(\mathcal{A})\operatorname{tr}(E^2)$, where $k(\mathcal{A})\le\Lambda_p/\lambda_p$ is the condition number. This loss cannot be avoided in general. For $N=2$, $\mathcal{A}=\mathrm{diag}(1,k)$ and $T=\big(\begin{smallmatrix}0&1\\1&0\end{smallmatrix}\big)$ one finds $\operatorname{tr}(E^2)=2$ and $|E|^2=k+k^{-1}$. Hence any argument that bounds $\operatorname{tr}(E^2)$ from below by a multiple of $|E|^2$ can only use a factor of order $1/k(\mathcal{A})$.
\end{remark}

Define the operator
\[
\mathcal{L}g=\operatorname{div}\big(w^{2-N}\mathcal{A}\nabla g\big),
\]
acting on $g\in W^{1,2}_{\rm loc}(\Omega_{\rm reg})$ in the sense of distributions. It is a linear operator in divergence form whose coefficients are continuous and uniformly elliptic on compact subsets of $\Omega_{\rm reg}$.

\begin{proposition}[Weak Bochner identity]\label{prop_bochner}
In $\mathcal{D}'(\Omega_{\rm reg})$,
\[
\mathcal{L}P=N(p-1)\,w^{1-N}|T|^2\ \ge 0.
\]
In particular $\mathcal{L}P$ is a nonnegative function in $L^1_{\rm loc}(\Omega_{\rm reg})$.
\end{proposition}

\begin{proof}
Let $D\Subset\Omega_{\rm reg}$. By Lemma \ref{lem_reg}, $X\in W^{1,2}(D)\cap C(\overline D)$ and $P\in W^{1,2}(D)\cap C(\overline D)$.

We first show that, in $\mathcal{D}'(D)$,
\begin{equation}\label{eq_divEX}
\operatorname{div}(EX)=\operatorname{tr}(E^2)+\frac{N-1}{N}X\cdot\nabla P.
\end{equation}
Let $X^\varepsilon$ be a mollification of $X$ on $D'\Subset D$. For smooth fields, $\partial_i\big(X^\varepsilon_j\partial_jX^\varepsilon_i\big)=\partial_iX^\varepsilon_j\,\partial_jX^\varepsilon_i+X^\varepsilon\cdot\nabla\operatorname{div}X^\varepsilon$, with summation over repeated indices. As $\varepsilon\to 0$ we have $X^\varepsilon\to X$ uniformly, $DX^\varepsilon\to DX$ in $L^2$, and $\nabla\operatorname{div}X^\varepsilon=(\nabla P)^\varepsilon\to\nabla P$ in $L^2$. Hence every term converges in $L^1(D')$ and
\[
\partial_i\big(X_j\partial_jX_i\big)=\operatorname{tr}\big((DX)^2\big)+X\cdot\nabla P\qquad\text{in }\mathcal{D}'(D').
\]
The product rule for $PX\in W^{1,1}(D')$ gives $\operatorname{div}(PX)=\nabla P\cdot X+P^2$. Since $(EX)_i=X_j\partial_jX_i-\frac PNX_i$ and $\operatorname{tr}(E^2)=\operatorname{tr}((DX)^2)-\frac{P^2}{N}$ (because $\operatorname{tr}DX=P$), we obtain \eqref{eq_divEX}.

By \eqref{eq_EX}, the $L^2$ field $w\mathcal{A}\nabla P$ equals $N(p-1)EX$ almost everywhere, so
\[
\operatorname{div}\big(w\mathcal{A}\nabla P\big)=N(p-1)\operatorname{tr}(E^2)+(N-1)(p-1)X\cdot\nabla P\qquad\text{in }\mathcal{D}'(D).
\]
For $G\in L^1_{\rm loc}$ with $\operatorname{div}G\in L^1_{\rm loc}$ and $\sigma\in C^1$, one has $\operatorname{div}(\sigma G)=\sigma\operatorname{div}G+\nabla\sigma\cdot G$ in $\mathcal{D}'$, as one sees by testing with $\sigma\varphi$. We apply this with $\sigma=w^{1-N}$ and $G=w\mathcal{A}\nabla P$. By \eqref{eq_chain}, $\nabla(w^{1-N})\cdot w\mathcal{A}\nabla P=(1-N)w^{1-N}\mathcal{A}\nabla w\cdot\nabla P=(1-N)(p-1)w^{1-N}X\cdot\nabla P$. The two terms containing $X\cdot\nabla P$ cancel, and Lemma \ref{lem_algebra}(i) gives the claim.
\end{proof}

\begin{remark}[Why $C^2$ regularity of the anisotropy is sufficient]\label{rem_C2}
The proof above uses the anisotropy at exactly one differential level beyond $A$ itself. On the noncritical set, $H\in C^2(\mathbb{R}^N\setminus\{0\})$ makes $A=\nabla(H^p/p)$ a $C^1$ map and $DA$ a continuous positive-definite matrix field. Difference quotients then give $w\in W^{2,2}_{\rm loc}$ and $X=A(\nabla w)\in W^{1,2}_{\rm loc}$, which is sufficient for the distributional Bochner identity. Neither Proposition \ref{prop_bochner} nor the later linearizations differentiate $DA$. Thus no third derivative of $H$ is used.
\end{remark}

\section{The sharp bound for the pressure}\label{sec_Pbound}

\begin{proposition}\label{prop_Pbound}
Let $V\in W^{1,p}_{\rm loc}(\mathbb{R}^N)$ be a weak solution of $-\Delta_p^{\check H}V=V^{p^*-1}$ in $\mathbb{R}^N$ with $0<V\le 1$, and let $w$, $P$ be as in \eqref{eq_transform}. Then $P\le c$ in $\mathbb{R}^N$.
\end{proposition}

\begin{proof}
Note that $w\ge 1$ and $P>0$.

\emph{Step 1: a priori bounds.} We claim that there is a structural constant $C_*$ with
\begin{equation}\label{eq_gradbound}
H(\nabla w)\le C_*w,\qquad P\le C_*\big(1+w^{p-1}\big)\qquad\text{in }\mathbb{R}^N.
\end{equation}
Fix $x\in\mathbb{R}^N$ and let $S_x=\sup_{B_2(x)}V$. The function $\tilde V=V/S_x$ satisfies $-\Delta_p^{\check H}\tilde V=\mathcal{V}\tilde V^{p-1}$ with $\mathcal{V}=V^{p^*-p}\le 1$ and $0<\tilde V\le 1$ on $B_2(x)$. Lemma \ref{lem_harnack}(ii) applied on $B_8(x)$ gives $S_x\le C\inf_{B_2(x)}V\le CV(x)$, and Lemma \ref{lem_harnack}(iii) gives $|\nabla\tilde V(x)|\le C$. Hence $|\nabla V|\le CV$. Since $\nabla w=-a^{-1}(w/V)\nabla V$, Lemma \ref{lem_norm}(a) gives $H(\nabla w)\le M_Ha^{-1}w|\nabla V|/V\le Cw$. The bound for $P$ follows from \eqref{eq_transform} and $w\ge 1$.

\emph{Step 2: the multiplier.} Fix $\ell>c$ and a parameter $m>0$, to be chosen in Step 5. For $t\ge 1$ define
\[
k(t)=\frac ct,\qquad Y(t)=m\,t^{-a}\Big(\frac{\ell-k(t)}{\ell-c}\Big)^{b},\qquad \theta(t)=\frac{N-1}{1+Y(t)},\qquad s(t)=\theta(t)\Big(1-\frac{\theta(t)}{N-1}\Big),
\]
\[
\phi(t)=\exp\Big(-\int_1^t\frac{\theta(\tau)}{\tau}\,d\tau\Big),\qquad g(t)=\frac{a\,s(t)-t\theta'(t)}{b}.
\]
Since $\ell-k\ge\ell-c>0$, all these functions are smooth on $[1,\infty)$, $Y>0$, $0<\theta<N-1$, $s>0$, $0<\phi\le 1$ and $t\phi'(t)=-\theta(t)\phi(t)$. A direct differentiation shows that $Y$ solves
\[
Y'+\Big(\frac at+\frac{b\,k'}{\ell-k}\Big)Y=0.
\]
From $\theta=(N-1)/(1+Y)$ we get $\theta'=-s\,Y'/Y$, and $tk'=-k$ gives $t\theta'=s\big(a-\frac{bk}{\ell-k}\big)$. Therefore
\begin{equation}\label{eq_g}
g=\frac{s\,k}{\ell-k}>0,\qquad g\,(\ell-k)=s\,k.
\end{equation}

\emph{Step 3: a differential inequality.} Let
\[
M=P\,\phi(w),\qquad\mathcal{O}=\{x\in\mathbb{R}^N:M(x)>\ell\}.
\]
The set $\mathcal{O}$ is open because $M$ is continuous. On $\mathcal{O}$ we have $P>\ell$ because $\phi\le 1$, and therefore
\begin{equation}\label{eq_h_lower}
h=\frac{wP-c}{b}\ge\frac{P-c}{b}>\frac{\ell-c}{b}>0.
\end{equation}
Hence $\mathcal{O}\subset\Omega_{\rm reg}$, and Lemma \ref{lem_reg} and Proposition \ref{prop_bochner} apply on $\mathcal{O}$. Set
\[
\zeta=\frac{w}{\phi(w)\,h}X\cdot\nabla M,\qquad\mathbf{d}=2(p-1)\theta(w)w^{1-N}\Big(\frac{wP}{(N-1)h}-1\Big)X.
\]
The drift $\mathbf{d}$ is continuous on $\mathcal{O}$ by \eqref{eq_h_lower}. We claim that, almost everywhere in $\mathcal{O}$,
\begin{equation}\label{eq_multiplier}
\mathcal{L}M-\mathbf{d}\cdot\nabla M\ge(p-1)\phi(w)w^{1-N}\Big(\frac{\zeta^2}{N-1}+g(w)P(P-\ell)\Big),
\end{equation}
where $\mathcal{L}M$ is a locally integrable function.

We first compute $\mathcal{L}M$. We have $w^{2-N}\mathcal{A}\nabla M=\phi(w)\,w^{2-N}\mathcal{A}\nabla P+(p-1)P\,\phi'(w)\,w^{2-N}X$ by \eqref{eq_chain}. The first field has divergence $\phi\,\mathcal{L}P+(p-1)\phi'w^{2-N}X\cdot\nabla P$, by the product rule used in the proof of Proposition \ref{prop_bochner}. The second belongs to $W^{1,2}_{\rm loc}$ and its divergence is computed with \eqref{eq_chain} and $\operatorname{div}X=P$. Using $\phi'=-\theta\phi/w$ and
\[
\frac{d}{dt}\big(t^{2-N}\phi'(t)\big)=\phi(t)\,t^{-N}\big(-t\theta'+\theta^2+(N-1)\theta\big),
\]
we get
\[
\frac{\mathcal{L}M}{(p-1)\phi(w)w^{1-N}}=N|T|^2-2\theta X\cdot\nabla P+P\Big[\big(-w\theta'+\theta^2+(N-1)\theta\big)\frac hw-\theta P\Big],
\]
where $\theta,\theta'$ are evaluated at $w$. Since $\nabla M=\phi(w)\big(\nabla P-\theta P\nabla w/w\big)$ and $X\cdot\nabla w=h$, we have $X\cdot\nabla P=\frac hw(\zeta+\theta P)$. The Kato inequality of Lemma \ref{lem_algebra} becomes $N|T|^2\ge(\zeta+\theta P)^2/(N-1)$. Moreover
\[
\frac{\mathbf{d}\cdot\nabla M}{(p-1)\phi(w)w^{1-N}}=2\theta\Big(\frac{P}{N-1}-\frac hw\Big)\zeta.
\]
Writing $k=c/w$, so that $h/w=(P-k)/b$ by \eqref{eq_transform}, and using $a+b=N-1$ and $s=\theta(1-\theta/(N-1))$, a direct expansion gives the identity
\begin{multline*}
\frac{(\zeta+\theta P)^2}{N-1}-2\theta\frac hw(\zeta+\theta P)+P\Big[\big(-w\theta'+\theta^2+(N-1)\theta\big)\frac hw-\theta P\Big]-2\theta\Big(\frac P{N-1}-\frac hw\Big)\zeta\\
=\frac{\zeta^2}{N-1}+P\Big(\frac{as-w\theta'}{b}(P-k)-s\,k\Big).
\end{multline*}
(All the terms linear in $\zeta$ cancel.) By \eqref{eq_g}, the last bracket equals $g(P-k)-sk=g(P-\ell)$, and \eqref{eq_multiplier} follows. The right-hand side of \eqref{eq_multiplier} is continuous and strictly positive on $\mathcal{O}$, because $g>0$ and $P>\ell$.

\emph{Step 4: behavior as $w\to\infty$.} Since $Y(t)\le m\,t^{-a}(\ell/(\ell-c))^b$, we have $\theta(t)\ge N-1-C_{\ell,m}t^{-a}$, hence $\phi(t)\le C_{\ell,m}t^{-(N-1)}$ for $t\ge 1$. By \eqref{eq_gradbound} and $p<N$,
\begin{equation}\label{eq_Mdecay}
0<M\le C_{\ell,m}\big(1+w^{p-1}\big)w^{1-N}\longrightarrow 0\qquad\text{as }w\to\infty.
\end{equation}
In particular $M$ is bounded on $\mathbb{R}^N$.

\emph{Step 5: conclusion.} Suppose that $P(x_0)>\ell$ for some $x_0$, and put $t_0=w(x_0)$. For $1\le t\le t_0$ we have $Y(t)\ge mt_0^{-a}$, so $\theta(t)\le(N-1)/(1+mt_0^{-a})$ and $\phi(t_0)\ge t_0^{-(N-1)/(1+mt_0^{-a})}$. We choose $m$ so large that $M(x_0)>\ell$, and we keep this $m$ from now on. Let $S=\sup_{\mathbb{R}^N}M\in(\ell,\infty)$ and let $x_k$ satisfy $M(x_k)\to S$. By \eqref{eq_Mdecay}, $w(x_k)$ is bounded, that is $V(x_k)\ge\delta_0>0$.

The translates $V_k=V(\cdot+x_k)$ solve the same equation, $0<V_k\le 1$ and $V_k(0)\ge\delta_0$. By Lemma \ref{lem_harnack}(ii),(iii), a subsequence converges in $C^1_{\rm loc}(\mathbb{R}^N)$ to a weak solution $V_\infty$ with $0<V_\infty\le 1$. The associated $w_k,P_k,M_k$ are continuous functions of $(V_k,\nabla V_k)$ on $\{V_k>0\}$, so they converge locally uniformly to the corresponding $w_\infty,P_\infty,M_\infty$. Hence $M_\infty\le S$ in $\mathbb{R}^N$ and $M_\infty(0)=S>\ell$.

Apply Step 3 to $V_\infty$. There is a ball $B$ centered at $0$ with $\overline B\subset\mathcal{O}_\infty$. On $B$, $M_\infty\in W^{1,2}(B)$ is a weak subsolution of the linear operator $g\mapsto\mathcal{L}_\infty g-\mathbf{d}_\infty\cdot\nabla g$, whose principal coefficients $w_\infty^{2-N}\mathcal{A}_\infty$ are continuous and uniformly elliptic on $\overline B$ and whose drift is bounded. It attains its maximum over $B$ at the center. The strong maximum principle \cite[Theorem 8.19]{GT01} gives $M_\infty\equiv S$ on $B$. Then $\nabla M_\infty=0$ on $B$ and the left-hand side of \eqref{eq_multiplier} vanishes, which contradicts the strict positivity of its right-hand side. Therefore $P\le\ell$. Since $\ell>c$ was arbitrary, $P\le c$.
\end{proof}

\begin{remark}
The multiplier is the pure-power case of the target-dependent multiplier of Zhang \cite{Zha26}. For the pure critical power the auxiliary function $Y$ is explicit and the decay rate of $\phi$ is $N-1>p-1$, so the sharp target $c$ is reached directly. The identity leading to \eqref{eq_multiplier} involves only $w$, $P$, $h$, $\theta$ and $X\cdot\nabla P$, and no comparison between $H$ and the Euclidean norm enters.
\end{remark}

\section{Rigidity of bounded profiles}\label{sec_profile}

\begin{lemma}\label{lem_root}
Let $f\in C^1(\mathbb{R}^N)$ with $f\ge 0$, $f(x_0)=0$ and $b\,H(\nabla f)^p\le c\,f$ in $\mathbb{R}^N$. Then $f(x)\le K\,H_0(x-x_0)^{p'}$ for all $x$.
\end{lemma}

\begin{proof}
For $\varepsilon>0$ the function $\Psi_\varepsilon=(f+\varepsilon)^{1/p'}$ is $C^1$, and $\nabla\Psi_\varepsilon=\frac{1}{p'}(f+\varepsilon)^{-1/p}\nabla f$. Hence
\[
H(\nabla\Psi_\varepsilon)^p=(p')^{-p}\frac{H(\nabla f)^p}{f+\varepsilon}\le\frac{c}{b\,(p')^p}=\frac{c}{N(p')^{p-1}},
\]
since $bp'=N$. By Lemma \ref{lem_norm}(a),
\[
\Psi_\varepsilon(x)-\Psi_\varepsilon(x_0)=\int_0^1\nabla\Psi_\varepsilon\big(x_0+t(x-x_0)\big)\cdot(x-x_0)\,dt\le\Big(\frac{c}{N(p')^{p-1}}\Big)^{1/p}H_0(x-x_0).
\]
Letting $\varepsilon\to 0$ and raising to the power $p'$ gives $f(x)\le\big(c/(N(p')^{p-1})\big)^{1/(p-1)}H_0(x-x_0)^{p'}=K\,H_0(x-x_0)^{p'}$.
\end{proof}

\begin{lemma}[Wulff tangent rigidity]\label{lem_tangent}
Let $\mathcal{T}\in C^1(\mathbb{R}^N)$ satisfy $\mathcal{T}\ge 0$, $\mathcal{T}(0)=0$, $b\,H(\nabla\mathcal{T})^p\le c\,\mathcal{T}$, and $\operatorname{div}A(\nabla\mathcal{T})=c$ weakly in $\mathbb{R}^N$. Then $\mathcal{T}=K H_0^{p'}$.
\end{lemma}

\begin{proof}
By Lemma \ref{lem_root}, $\mathcal{T}\le KH_0^{p'}$. The constants satisfy
\begin{equation}\label{eq_Kidentity}
(Kp')^{p-1}=\frac cN,\qquad b\,(Kp')^p=c\,K.
\end{equation}
By Lemma \ref{lem_norm}(c),
\[
H_0\big(A(\nabla\mathcal{T})\big)=H(\nabla\mathcal{T})^{p-1}\le\Big(\frac cb\,\mathcal{T}\Big)^{1/p'}\le\Big(\frac{cK}{b}\Big)^{1/p'}H_0=(Kp')^{p-1}H_0=\frac cN\,H_0.
\]
Fix $R>0$. By Lemma \ref{lem_flux}, $\int_{\partial\mathcal{W}_R}A(\nabla\mathcal{T})\cdot\nu\,dS=c|\mathcal{W}_R|=c|\mathcal{W}_1|R^N$. On the other hand, by Lemma \ref{lem_norm}(a),(e) and the last display, on $\partial\mathcal{W}_R$
\[
A(\nabla\mathcal{T})\cdot\nu\le H_0\big(A(\nabla\mathcal{T})\big)H(\nu)\le\frac cN\,\frac{R}{|\nabla H_0|},
\]
and the integral of the right-hand side over $\partial\mathcal{W}_R$ equals $c|\mathcal{W}_1|R^N$. Both inequalities are between continuous functions, so both are equalities at every point of $\partial\mathcal{W}_R$. Since $R$ is arbitrary, for every $x\neq 0$ we have $H_0(A(\nabla\mathcal{T}(x)))=\frac cNH_0(x)>0$, and equality holds in $A(\nabla\mathcal{T})\cdot\nu\le H_0(A(\nabla\mathcal{T}))H(\nu)$. The equality case of Lemma \ref{lem_norm}(a) and Lemma \ref{lem_norm}(b) give
\[
A(\nabla\mathcal{T}(x))=H_0\big(A(\nabla\mathcal{T}(x))\big)\nabla H(\nu)=\frac cNH_0(x)\,\nabla H(\nabla H_0(x))=\frac cN\,x.
\]
By Lemma \ref{lem_norm}(c) and \eqref{eq_Kidentity}, $\nabla\mathcal{T}=A^{-1}(\frac cNx)=(Kp')H_0^{p'-1}\nabla H_0=\nabla(KH_0^{p'})$. Since $\mathcal{T}(0)=0$, the claim follows.
\end{proof}

\begin{proposition}\label{prop_profile}
Let $V\in W^{1,p}_{\rm loc}(\mathbb{R}^N)$ be a weak solution of $-\Delta_p^{\check H}V=V^{p^*-1}$ in $\mathbb{R}^N$ with $0<V\le 1$ and $V(0)=1$. Then $V=(1+KH_0^{p'})^{-a}$.
\end{proposition}

\begin{proof}
Let $w$, $X$, $h$, $P$, $\Omega_{\rm reg}$ be defined by \eqref{eq_transform} with $\Omega=\mathbb{R}^N$. Then $w\ge 1$, $w(0)=1$, $\nabla w(0)=0$. Set $\Pi=c-P$, a continuous function.

\emph{Step 1.} By Proposition \ref{prop_Pbound}, $\Pi\ge 0$, which by \eqref{eq_transform} reads
\begin{equation}\label{eq_Pi}
\Pi=\frac{c(w-1)-b\,h}{w}\ge 0,\qquad b\,H(\nabla w)^p\le c\,(w-1)\quad\text{in }\mathbb{R}^N.
\end{equation}
By Proposition \ref{prop_bochner}, $\mathcal{L}\Pi=-\mathcal{L}P\le 0$ in $\mathcal{D}'(\Omega_{\rm reg})$. If $x_1$ is a critical point of $w$, then $P(x_1)=c/w(x_1)$.

\emph{Step 2: tangent at a point where $w=1$.} Let $x_0$ satisfy $w(x_0)=1$. By \eqref{eq_Pi} and Lemma \ref{lem_root}, $w(x)-1\le KH_0(x-x_0)^{p'}$. For $r>0$ let
\[
\mathcal{T}_r(y)=\frac{w(x_0+ry)-1}{r^{p'}}.
\]
Then $0\le\mathcal{T}_r\le KH_0^{p'}$, $\mathcal{T}_r(0)=0$, and $bH(\nabla\mathcal{T}_r)^p\le c\mathcal{T}_r$ because $(1-p')p=-p'$. Since $(p'-1)(p-1)=1$, we have $A(\nabla\mathcal{T}_r)(y)=r^{-1}X(x_0+ry)$, so $\operatorname{div}A(\nabla\mathcal{T}_r)=P(x_0+r\,\cdot)$ weakly, with $0<P\le c$. By Lemma \ref{lem_harnack}(iii), the family $\{\mathcal{T}_r\}$ is precompact in $C^1_{\rm loc}(\mathbb{R}^N)$. Since $P$ is continuous and $P(x_0)=c$, every limit as $r\to 0$ satisfies the hypotheses of Lemma \ref{lem_tangent}, so $\mathcal{T}_r\to KH_0^{p'}$ in $C^1_{\rm loc}(\mathbb{R}^N)$. Using $H(\nabla H_0)=1$ and \eqref{eq_Kidentity}, we obtain, uniformly for $H_0(y)=1$,
\begin{equation}\label{eq_tangent_expansion}
w(x_0+ry)-1=r^{p'}(K+o(1)),\qquad h(x_0+ry)=r^{p'}\Big(\frac{cK}{b}+o(1)\Big),\qquad \Pi(x_0+ry)=o(r^{p'})
\end{equation}
as $r\to 0$. Moreover $\nabla(KH_0^{p'})$ does not vanish on $\partial\mathcal{W}_1$, so there is $r_0>0$ with $\mathcal{W}_{r_0}(x_0)\setminus\{x_0\}\subset\Omega_{\rm reg}$.

\emph{Step 3: $\Pi\equiv 0$ near a point where $w=1$.} Let $x_0$ be as in Step 2, fix $\sigma$ with $\max\{0,1-b\}<\sigma<1$, and let $\psi=(w-1)^\sigma$ on $\mathcal{W}_{r_0}(x_0)\setminus\{x_0\}$, where $w>1$. Since $w^{2-N}\mathcal{A}\nabla\psi=(p-1)\sigma w^{2-N}(w-1)^{\sigma-1}X\in W^{1,2}_{\rm loc}$, the product rule and $\operatorname{div}X=P$ give
\[
\mathcal{L}\psi=(p-1)\sigma\,w^{2-N}(w-1)^{\sigma-2}\Big[(w-1)P+\Big(\sigma-1+(2-N)\frac{w-1}{w}\Big)h\Big].
\]
By \eqref{eq_tangent_expansion}, the bracket equals $r^{p'}\big(cK(1+\frac{\sigma-1}{b})+o(1)\big)$ on $\partial\mathcal{W}_r(x_0)$, and $1+(\sigma-1)/b>0$ by the choice of $\sigma$. Hence there is $r_1\in(0,r_0)$ such that $\mathcal{L}\psi$ is a positive continuous function on $\mathcal{W}_{r_1}(x_0)\setminus\{x_0\}$.

Suppose that $\Pi$ does not vanish identically on $\mathcal{W}_{r_1}(x_0)\setminus\{x_0\}$. This set is connected, $\mathcal{L}$ has continuous coefficients which are uniformly elliptic on its compact subsets, $\Pi\ge 0$ and $\mathcal{L}\Pi\le 0$. The strong maximum principle \cite[Theorem 8.19]{GT01} gives $\Pi>0$ there. Fix $r_2\in(0,r_1)$, let $m_2=\min_{\partial\mathcal{W}_{r_2}(x_0)}\Pi>0$, and choose $\lambda_0>0$ with $\lambda_0\max_{\partial\mathcal{W}_{r_2}(x_0)}\psi\le m_2/2$. For $0<\varepsilon<r_2/2$ put $M_\varepsilon=\max_{\partial\mathcal{W}_\varepsilon(x_0)}\psi$ and $\mathcal{H}_\varepsilon=\Pi-\lambda_0(\psi-M_\varepsilon)$ on the annulus $\mathcal{D}_\varepsilon=\mathcal{W}_{r_2}(x_0)\setminus\overline{\mathcal{W}_\varepsilon(x_0)}$. Then $\mathcal{H}_\varepsilon\ge 0$ on $\partial\mathcal{D}_\varepsilon$ and $\mathcal{L}\mathcal{H}_\varepsilon\le 0$ in $\mathcal{D}_\varepsilon$, so the weak maximum principle \cite[Theorem 8.1]{GT01} gives $\mathcal{H}_\varepsilon\ge 0$ in $\mathcal{D}_\varepsilon$. As $\varepsilon\to 0$, $M_\varepsilon\to 0$, hence
\[
\Pi\ge\lambda_0(w-1)^\sigma\qquad\text{in }\mathcal{W}_{r_2}(x_0)\setminus\{x_0\}.
\]
By \eqref{eq_tangent_expansion} the right-hand side is at least $\lambda_0(K/2)^\sigma r^{\sigma p'}$ on $\partial\mathcal{W}_r(x_0)$ for small $r$, while $\Pi=o(r^{p'})$ there. Since $\sigma<1$ this is impossible. Therefore $\Pi\equiv 0$ in $\mathcal{W}_{r_1}(x_0)$, using also $\Pi(x_0)=0$.

\emph{Step 4: $P\equiv c$.} The set $\mathcal{N}=\{\Pi=0\}$ is closed and contains $0$. Let $x_1\in\mathcal{N}$. If $\nabla w(x_1)\neq 0$, then $\Pi\ge 0$, $\mathcal{L}\Pi\le 0$ on a ball $B\subset\Omega_{\rm reg}$ around $x_1$, and $\Pi(x_1)=0$, so \cite[Theorem 8.19]{GT01} gives $\Pi\equiv 0$ on $B$. If $\nabla w(x_1)=0$, then $c=P(x_1)=c/w(x_1)$, so $w(x_1)=1$ and Step 3 applies at $x_1$. Hence $\mathcal{N}$ is open, and $\mathcal{N}=\mathbb{R}^N$.

\emph{Step 5: identification.} By Proposition \ref{prop_bochner} and Step 4, $N(p-1)w^{1-N}|T|^2=\mathcal{L}c=0$, so $T=0$ and $E=0$ almost everywhere in $\Omega_{\rm reg}$. Thus $DX=\frac cN\mathrm{Id}$ almost everywhere in $\Omega_{\rm reg}$.

The complement $\mathcal{C}=\mathbb{R}^N\setminus\Omega_{\rm reg}=\{X=0\}$ has empty interior: on an open set contained in $\mathcal{C}$ we would have $\operatorname{div}X=0$, while $\operatorname{div}X=P=c>0$ by Lemma \ref{lem_transformed}. In particular $\Omega_{\rm reg}\neq\emptyset$. Let $\mathcal{U}$ be a connected component of $\Omega_{\rm reg}$. Since $X\in W^{1,1}_{\rm loc}(\mathcal{U})$ and $D\big(X-\frac cNx\big)=0$ almost everywhere on $\mathcal{U}$, there is $x_{\mathcal{U}}\in\mathbb{R}^N$ with
\[
X(x)=\frac cN(x-x_{\mathcal{U}})\qquad\text{on }\mathcal{U}.
\]
Let $y\in\partial\mathcal{U}$. Then $y\notin\Omega_{\rm reg}$, because the component of $\Omega_{\rm reg}$ containing $y$ would be an open neighborhood of $y$ disjoint from $\mathcal{U}$. Hence $X(y)=0$, and by continuity $X(y)=\frac cN(y-x_{\mathcal{U}})$, so $y=x_{\mathcal{U}}$. Thus $\partial\mathcal{U}\subset\{x_{\mathcal{U}}\}$, and $x_{\mathcal{U}}\notin\mathcal{U}$ because $X(x_{\mathcal{U}})=0$. Therefore $\mathcal{U}$ is a nonempty subset of $\mathbb{R}^N\setminus\{x_{\mathcal{U}}\}$ which is open and relatively closed in it. Since $N\ge 2$ this set is connected, so $\mathcal{U}=\mathbb{R}^N\setminus\{x_{\mathcal{U}}\}$. By continuity $X(x)=\frac cN(x-x_{\mathcal{U}})$ on $\mathbb{R}^N$, and $X(0)=A(\nabla w(0))=0$ gives $x_{\mathcal{U}}=0$.

As in the proof of Lemma \ref{lem_tangent}, $\nabla w=A^{-1}(\frac cNx)=\nabla(KH_0^{p'})$, and $w(0)=1$ gives $w=1+KH_0^{p'}$. Hence $V=w^{-a}=(1+KH_0^{p'})^{-a}$.
\end{proof}

\begin{proof}[Proof of Theorem \ref{thm_profile}]
Let $H$ be the norm of Theorem \ref{thm_profile}. The reflected norm $\hat H$ satisfies (H), and its own reflection is $H$. Applying Proposition \ref{prop_profile} to the norm $\hat H$, whose dual norm is $\hat H_0$, gives $V=(1+K\hat H_0^{p'})^{-a}=U$.
\end{proof}

\section{Isolated singularities}\label{sec_pole}

From now on $H$ is again the norm of Theorem \ref{thm_main}. In this section $W$ is as in Theorem \ref{thm_pole}, and all Wulff balls are those of $\hat H_0$, with respect to which the fundamental solution $\mu=\hat H_0^{-\alpha}$ is radial. For $0<r<R$ let
\[
W_r(y)=r^\alpha W(ry),\qquad 0<\hat H_0(y)<R/r.
\]
Since $(\alpha+1)(p-1)=N-1$, one has $A(\nabla W_r)(y)=r^{N-1}A(\nabla W)(ry)$. Hence $W_r$ is a weak solution of $\Delta_p^HW_r=0$, $W_r\le C_0\mu$, and, by Lemma \ref{lem_norm}(e) for $\hat H_0$ and the change of variables $x=ry$ on Wulff spheres,
\[
\int_{\partial\hat{\mathcal{W}}_t}A(\nabla W_r)\cdot\nu\,dS=\int_{\partial\hat{\mathcal{W}}_{rt}}A(\nabla W)\cdot\nu\,dS=\mathcal{F}\qquad(0<t<R/r).
\]
The flux $\mathcal{F}$ is independent of the radius by Lemma \ref{lem_flux} applied to $G=A(\nabla W)$, $g=0$.

\begin{lemma}\label{lem_blowup_pole}
Every sequence $r_j\to 0$ has a subsequence along which $W_{r_j}\to Z$ in $C^1_{\rm loc}(\mathbb{R}^N\setminus\{0\})$, where $Z\in C^1(\mathbb{R}^N\setminus\{0\})$ is a weak solution of $\Delta_p^HZ=0$ with $0\le Z\le C_0\mu$ and $\int_{\partial\hat{\mathcal{W}}_t}A(\nabla Z)\cdot\nu\,dS=\mathcal{F}$ for all $t>0$.
\end{lemma}

\begin{proof}
On every compact subset of $\mathbb{R}^N\setminus\{0\}$ the functions $W_{r_j}$ are eventually defined and bounded by $C_0\mu$. Lemma \ref{lem_harnack}(iii) gives uniform $C^{1,\gamma}$ bounds, and a diagonal argument gives a subsequence converging in $C^1_{\rm loc}(\mathbb{R}^N\setminus\{0\})$. The weak equation and the flux pass to the limit because $A(\nabla W_{r_j})\to A(\nabla Z)$ locally uniformly.
\end{proof}

\begin{proof}[Proof of Theorem \ref{thm_pole}]
\emph{Step 1: the quotient limit and (i).} Let $\kappa_+=\limsup_{x\to 0}W/\mu$ and $\kappa_-=\liminf_{x\to 0}W/\mu$, so $0\le\kappa_-\le\kappa_+\le C_0$.

If $\kappa_+=0$, then $W_r\to 0$ uniformly on compact subsets of $\mathbb{R}^N\setminus\{0\}$, every limit in Lemma \ref{lem_blowup_pole} is $Z\equiv 0$, and its flux is $0\neq\mathcal{F}$. Hence $\kappa_+>0$. Choose $x_j\to 0$ with $W(x_j)/\mu(x_j)\to\kappa_+$, and put $r_j=\hat H_0(x_j)$, $y_j=x_j/r_j\in\partial\hat{\mathcal{W}}_1$. Along a subsequence $y_j\to y_0$ and $W_{r_j}\to Z$ as in Lemma \ref{lem_blowup_pole}. Since $\mu(y_j)=1$ and $\mu(x_j)=r_j^{-\alpha}$,
\[
Z(y_0)=\lim_jW_{r_j}(y_j)=\lim_j\frac{W(x_j)}{\mu(x_j)}=\kappa_+\mu(y_0),
\]
and for every $y\neq 0$, $Z(y)=\lim_j\frac{W}{\mu}(r_jy)\,\mu(y)\le\kappa_+\mu(y)$. At a point where $Z=\kappa_+\mu$, the gradients of $Z$ and $\kappa_+\mu$ coincide and do not vanish. Lemma \ref{lem_strongcomp} on the connected set $\mathbb{R}^N\setminus\{0\}$ gives $Z\equiv\kappa_+\mu$. By Lemma \ref{lem_bubble}(c),
\begin{equation}\label{eq_flux_kappa}
\mathcal{F}=-N|\hat{\mathcal{W}}_1|(\alpha\kappa_+)^{p-1}<0.
\end{equation}
Now choose $x_j\to 0$ with $W(x_j)/\mu(x_j)\to\kappa_-$ and a limit $Z'$ as before. Then $Z'\ge\kappa_-\mu$ with equality at some $y_0'\in\partial\hat{\mathcal{W}}_1$. If $\kappa_-=0$, the nonnegative solution $Z'$ vanishes at $y_0'$, so $Z'\equiv 0$ by the strong minimum principle, and its flux is zero, a contradiction. Hence $\kappa_->0$, Lemma \ref{lem_strongcomp} applied to $Z'\ge\kappa_-\mu$ gives $Z'\equiv\kappa_-\mu$, and comparing fluxes with \eqref{eq_flux_kappa} gives $\kappa_-=\kappa_+$. Set $\kappa=\kappa_+$. Then $W/\mu\to\kappa$ as $x\to 0$, (i) holds, and every limit in Lemma \ref{lem_blowup_pole} equals $\kappa\mu$. Consequently
\begin{equation}\label{eq_C1_pole}
W_r\to\kappa\mu\quad\text{in }C^1_{\rm loc}(\mathbb{R}^N\setminus\{0\})\quad\text{as }r\to 0.
\end{equation}

\emph{Step 2: the remainder is bounded.} Put $\Phi=\kappa\mu$, $\bar z=W-\Phi$ and $\varepsilon(r)=\max_{\partial\hat{\mathcal{W}}_r}|W/\mu-\kappa|$, so $\varepsilon(r)\to 0$. Fix $R_0\in(0,R)$ and let $C_+=\max\{0,\max_{\partial\hat{\mathcal{W}}_{R_0}}\bar z\}$ and $C_-=\min\{0,\min_{\partial\hat{\mathcal{W}}_{R_0}}\bar z\}$. For $r$ small, $\kappa-\varepsilon(r)>0$, and the functions $(\kappa\pm\varepsilon(r))\mu+C_\pm$ are $C^1$ weak solutions of $\Delta_p^H=0$ in $\hat{\mathcal{W}}_{R_0}\setminus\overline{\hat{\mathcal{W}}_r}$ by Lemma \ref{lem_bubble}(c). They lie above and below $W$ on both boundary components. Lemma \ref{lem_comparison} gives $(\kappa-\varepsilon(r))\mu+C_-\le W\le(\kappa+\varepsilon(r))\mu+C_+$ in this annulus. Fixing $x$ and letting $r\to 0$,
\[
C_-\le\bar z\le C_+\qquad\text{in }\hat{\mathcal{W}}_{R_0}\setminus\{0\}.
\]

\emph{Step 3: the extrema on spheres have limits.} Let $M(r)=\max_{\partial\hat{\mathcal{W}}_r}\bar z$ and $m(r)=\min_{\partial\hat{\mathcal{W}}_r}\bar z$ for $0<r<R_0$. If $0<r<t<s<R_0$, comparing $W$ with the solution $\Phi+\max\{M(r),M(s)\}$ on $\hat{\mathcal{W}}_s\setminus\overline{\hat{\mathcal{W}}_r}$ by Lemma \ref{lem_comparison} gives $M(t)\le\max\{M(r),M(s)\}$. We claim that $M$ has a limit at $0$. Otherwise $\ell_0=\liminf_{r\to 0}M<L_0=\limsup_{r\to 0}M$, and with $\delta=(L_0-\ell_0)/3$ we may choose successively $s_0$ with $M(s_0)<\ell_0+\delta$, then $t\in(0,s_0)$ with $M(t)>L_0-\delta$, then $r\in(0,t)$ with $M(r)<\ell_0+\delta$. This contradicts $M(t)\le\max\{M(r),M(s_0)\}$. By Step 2 the limit $\beta_+=\lim_{r\to 0}M(r)$ is finite. In the same way $m(t)\ge\min\{m(r),m(s)\}$, and $\beta_-=\lim_{r\to 0}m(r)$ exists, with $\beta_-\le\beta_+$.

\emph{Step 4: $\beta_-=\beta_+$.} By \eqref{eq_C1_pole} and the homogeneity $\nabla\Phi(ry)=r^{-\alpha-1}\nabla\Phi(y)$,
\begin{equation}\label{eq_er}
r^{\alpha+1}\nabla\bar z(ry)\to 0\qquad\text{uniformly on compact subsets of }\mathbb{R}^N\setminus\{0\}.
\end{equation}
Since $|\nabla\Phi|$ is bounded above and below by positive multiples of $\hat H_0^{-\alpha-1}$, \eqref{eq_er} implies the following quantitative separation from the origin: for every fixed compact annulus $K\Subset\mathbb{R}^N\setminus\{0\}$ there are $r_K>0$ and $c_K>0$ such that
\[
\big|\nabla\Phi(y)+t\,r^{\alpha+1}\nabla\bar z(ry)\big|\ge c_K\qquad(y\in K,\ 0<r<r_K,\ 0\le t\le1).
\]
Equivalently, after returning to the original variables, there is $\rho_0>0$ such that for $0<\hat H_0(x)<\rho_0$ the segment joining $\nabla\Phi(x)$ and $\nabla W(x)$ stays away from the origin. There
\[
A(\nabla W)-A(\nabla\Phi)=\mathcal{B}(x)\nabla\bar z,\qquad\mathcal{B}(x)=\int_0^1DA\big(\nabla\Phi(x)+t\nabla\bar z(x)\big)dt.
\]
For $r\to 0$ put $\bar z_r(y)=\bar z(ry)$ and $\mathcal{B}_r(y)=r^{(\alpha+1)(p-2)}\mathcal{B}(ry)$. Since $DA$ is homogeneous of degree $p-2$,
\[
\mathcal{B}_r(y)=\int_0^1DA\big(\nabla\Phi(y)+t\,r^{\alpha+1}\nabla\bar z(ry)\big)dt,
\]
and $\operatorname{div}(\mathcal{B}_r\nabla\bar z_r)=0$ weakly in $\{0<\hat H_0<\rho_0/r\}$, because $\mathcal{B}_r(y)\nabla\bar z_r(y)=r^{(\alpha+1)(p-2)+1}\big(A(\nabla W)-A(\nabla\Phi)\big)(ry)$. By \eqref{eq_er}, the quantitative separation above, and the continuity of $DA$ away from the origin, $\mathcal{B}_r\to\mathcal{B}_0=DA(\nabla\Phi)$ uniformly on compact subsets of $\mathbb{R}^N\setminus\{0\}$. Moreover, on every fixed compact annulus the arguments of $DA$ remain in a compact subset of $\mathbb{R}^N\setminus\{0\}$; hence \eqref{eq_ellipticity} gives ellipticity and boundedness constants for $\mathcal{B}_r$ that are uniform for all sufficiently small $r$. This observation is valid for both $p\ge2$ and the singular range $1<p<2$.

Let $r_j\to 0$. The functions $\bar z_{r_j}$ are bounded by Step 2. The De Giorgi--Nash--Moser estimate \cite[Theorem 8.24]{GT01} gives uniform H\"older bounds on compact annuli, and the Caccioppoli inequality gives uniform $L^2$ bounds for $\nabla\bar z_{r_j}$ there. Along a subsequence, $\bar z_{r_j}\to\tilde z$ locally uniformly in $\mathbb{R}^N\setminus\{0\}$ and $\nabla\bar z_{r_j}\rightharpoonup\nabla\tilde z$ weakly in $L^2_{\rm loc}$. Passing to the limit in the weak formulation gives $\operatorname{div}(\mathcal{B}_0\nabla\tilde z)=0$ in $\mathbb{R}^N\setminus\{0\}$. For every $t>0$, uniform convergence on $\partial\hat{\mathcal{W}}_t$ gives
\[
\max_{\partial\hat{\mathcal{W}}_t}\tilde z=\lim_jM(r_jt)=\beta_+,\qquad\min_{\partial\hat{\mathcal{W}}_t}\tilde z=\beta_-.
\]
Thus $\tilde z\le\beta_+$ in $\mathbb{R}^N\setminus\{0\}$ with equality at some point of $\partial\hat{\mathcal{W}}_1$. The coefficients $\mathcal{B}_0$ are continuous and uniformly elliptic on compact subsets of the connected set $\mathbb{R}^N\setminus\{0\}$, so \cite[Theorem 8.19]{GT01} gives $\tilde z\equiv\beta_+$, hence $\beta_-=\beta_+=:\beta$. Since $m(r)\le\bar z\le M(r)$ on $\partial\hat{\mathcal{W}}_r$, we obtain $\bar z(x)\to\beta$ as $x\to 0$, which is (ii).

\emph{Step 5: (iii).} Let $z=\bar z-\beta$ and $z_r(y)=z(ry)=\bar z_r(y)-\beta$. By (ii), $\sup_{1/2\le \hat H_0\le 4}|z_r|\to 0$. The function $z_r$ solves $\operatorname{div}(\mathcal{B}_r\nabla z_r)=0$ on $\{1/4<\hat H_0<8\}$, where the $\mathcal{B}_r$ are uniformly elliptic and bounded with constants independent of small $r$. Testing with $\eta^2z_r$, where $\eta$ is a cut-off equal to $1$ on $\{1\le \hat H_0\le 2\}$ and supported in $\{1/2<\hat H_0<4\}$, gives the Caccioppoli inequality
\[
\int_{\{1\le \hat H_0\le 2\}}|\nabla z_r|^2dy\le C\int_{\{1/2\le \hat H_0\le 4\}}|z_r|^2dy\to 0.
\]
The last part of (iii) is \eqref{eq_er}, since $\nabla z=\nabla\bar z$.

\emph{Step 6: (iv).} Let $\mathbf{P}=\mathbf{P}[W,0]$ be the field \eqref{eq_pohozaev_field} with $\mathcal{G}=0$, and $Q(r)=\int_{\partial\hat{\mathcal{W}}_r}\mathbf{P}\cdot\nu\,dS$. By Lemma \ref{lem_pohozaev} with $f=0$, $Q$ is constant on $(0,R)$. By the coarea formula and the change of variables $x=ry$,
\[
Q=\frac1r\int_r^{2r}Q(s)\,ds=\frac1r\int_{\{r<\hat H_0<2r\}}\mathbf{P}\cdot\nabla \hat H_0\,dx=r^{N-1}\int_{\{1<\hat H_0<2\}}\mathbf{P}(ry)\cdot\nabla \hat H_0(y)\,dy.
\]
Let $e_r(y)=r^{\alpha+1}\nabla z(ry)=r^\alpha\nabla z_r(y)$. Then $\nabla W(ry)=r^{-\alpha-1}\big(\nabla\Phi(y)+e_r(y)\big)$ and $W(ry)=r^{-\alpha}\Phi(y)+\beta+z_r(y)$. Using $r^{N-1}r^{-(\alpha+1)(p-1)}=1$ and $N-(\alpha+1)p=-\alpha$, a direct substitution gives
\[
r^{N-1}\mathbf{P}(ry)=r^{-\alpha}\mathbf{P}[\Phi,0](y)+a\,\beta\,A(\nabla\Phi(y))+\mathbf{R}_r(y),
\]
where
\begin{multline*}
\mathbf{R}_r=r^{-\alpha}\Big[\big(a\Phi+y\cdot\nabla\Phi\big)\big(A(\nabla\Phi+e_r)-A(\nabla\Phi)\big)+(y\cdot e_r)A(\nabla\Phi+e_r)-\frac yp\big(H(\nabla\Phi+e_r)^p-H(\nabla\Phi)^p\big)\Big]\\
+a\beta\big(A(\nabla\Phi+e_r)-A(\nabla\Phi)\big)+a\,z_r\,A(\nabla\Phi+e_r).
\end{multline*}
On $\{1\le \hat H_0\le 2\}$, $\nabla\Phi$ ranges in a compact subset of $\mathbb{R}^N\setminus\{0\}$ and $e_r\to 0$ uniformly, so for small $r$
\[
|A(\nabla\Phi+e_r)-A(\nabla\Phi)|+|H(\nabla\Phi+e_r)^p-H(\nabla\Phi)^p|\le C|e_r|.
\] Since $r^{-\alpha}|e_r|=|\nabla z_r|$, we get $|\mathbf{R}_r|\le C\big(|\nabla z_r|+|e_r|+|z_r|\big)$, and by (iii), $\int_{\{1<\hat H_0<2\}}|\mathbf{R}_r|\to 0$.

The field $\mathbf{P}[\Phi,0]$ vanishes identically. Indeed, by Lemma \ref{lem_bubble}(c),
\[
y\cdot\nabla\Phi=-\alpha\Phi,\qquad A(\nabla\Phi)=-(\alpha\kappa)^{p-1}\hat H_0^{-N}y,\qquad H(\nabla\Phi)^p=(\alpha\kappa)^p\hat H_0^{-(\alpha+1)p},
\]
and $(\alpha+1)p=\alpha+N$, so
\[
\mathbf{P}[\Phi,0](y)=(\alpha\kappa)^{p-1}\kappa\,\hat H_0^{-\alpha-N}\Big(\alpha-a-\frac\alpha p\Big)\,y=0,
\]
because $\alpha(p-1)/p=a$. Finally, by Lemma \ref{lem_bubble}(c), \eqref{eq_flux_kappa} and the coarea formula, $\int_{\{1<\hat H_0<2\}}a\beta A(\nabla\Phi)\cdot\nabla \hat H_0\,dy=a\beta\mathcal{F}$. Hence $Q=a\beta\mathcal{F}+o(1)$ as $r\to 0$, and since $Q$ is constant, $Q=a\beta\mathcal{F}$.
\end{proof}

\section{First contact and the Schoen-type estimate}\label{sec_schoen}

The argument of this section follows the point-selection and first-contact strategy of \cite{Zha26}, with the Wulff geometry of $H_0$ and $\hat H_0$ in place of Euclidean balls, Theorem \ref{thm_pole} in place of the isolated-singularity input, and Lemma \ref{lem_morrey} in place of the Wolff potential estimate. Throughout this section $q=p^*-1$. We use the identities
\begin{equation}\label{eq_exponents}
ap'=\alpha,\qquad \alpha(p-1)=N-p,\qquad \alpha q=N+p',\qquad q-p+1=\frac pa .
\end{equation}

\begin{proposition}[Point selection]\label{prop_selection}
Let $R_j>0$ and let $u_j$ be positive weak solutions of $-\Delta_p^Hu_j=u_j^q$ in $\mathcal{W}_{3R_j}$. Let $M_j=\max_{\overline{\mathcal{W}_{R_j}}}u_j$ and $m_j=\min_{\overline{\mathcal{W}_{2R_j}}}u_j$, and assume
\begin{equation}\label{eq_contra}
R_j^{N-p}M_jm_j^{p-1}\to\infty.
\end{equation}
Then there are $x_j\in\mathcal{W}_{2R_j}$ such that, with
\[
t_j=u_j(x_j),\quad D_j=2R_j-H_0(x_j),\quad\sigma_j=\frac{D_j}8,\quad\delta_j=t_j^{-1/a},\quad\Gamma_j=\frac{\sigma_j}{\delta_j},\quad\Xi_j=t_j\,\sigma_j^{N-p}m_j^{p-1},
\]
the following hold after passing to a subsequence.
\begin{enumerate}[label=\textup{(\roman*)}]
\item $\Xi_j\to\infty$ and $\Gamma_j^{N-p}=\Xi_j(t_j/m_j)^{p-1}\ge\Xi_j$.
\item $v_j(y)=t_j^{-1}u_j(x_j+\delta_jy)$ is a weak solution of $-\Delta_p^Hv_j=v_j^q$ in $\mathcal{W}_{4\Gamma_j}$ with $v_j(0)=1$ and $0<v_j\le 2^{N-p}$.
\item $v_j\to U$ in $C^1_{\rm loc}(\mathbb{R}^N)$.
\end{enumerate}
\end{proposition}

\begin{proof}
Let $d_j(x)=2R_j-H_0(x)$ and let $x_j$ maximize $d_j^{N-p}u_j$ over $\overline{\mathcal{W}_{2R_j}}$. The maximum is positive and $d_j$ vanishes on $\partial\mathcal{W}_{2R_j}$, so $x_j\in\mathcal{W}_{2R_j}$. If $y_j$ is a maximum point of $u_j$ on $\overline{\mathcal{W}_{R_j}}$, then $d_j(y_j)\ge R_j$, hence $t_jD_j^{N-p}\ge M_jR_j^{N-p}$, and \eqref{eq_contra} gives $\Xi_j=8^{p-N}t_jD_j^{N-p}m_j^{p-1}\to\infty$. Since $\delta_j^{-(N-p)}=t_j^p$, we have $\Gamma_j^{N-p}=\sigma_j^{N-p}t_j^p=\Xi_j(t_j/m_j)^{p-1}$, and $t_j\ge m_j$. This is (i).

If $H_0(x-x_j)<D_j/2$, the triangle inequality for $H_0$ gives $d_j(x)>D_j/2$, and maximality gives $u_j(x)\le(D_j/d_j(x))^{N-p}t_j<2^{N-p}t_j$. The equation is invariant under $u\mapsto t^{-1}u(x_j+t^{-1/a}\,\cdot)$, because $t^{1-p}t^{-p/a}t^q=1$ by \eqref{eq_exponents}. Since $\mathcal{W}_{D_j/2}(x_j)$ corresponds to $\mathcal{W}_{4\Gamma_j}$, this gives (ii).

For fixed $\rho>0$ and $j$ large, $H_0(y)\le\rho$ implies $d_j(x_j+\delta_jy)\ge D_j-\delta_j\rho=D_j(1-\rho/(8\Gamma_j))$, hence
\begin{equation}\label{eq_sharp_bound}
v_j(y)\le\Big(1-\frac{\rho}{8\Gamma_j}\Big)^{-(N-p)}\qquad(H_0(y)\le\rho).
\end{equation}
By Lemma \ref{lem_harnack}(ii),(iii), a subsequence converges in $C^1_{\rm loc}(\mathbb{R}^N)$ to a weak solution $V$ in $\mathbb{R}^N$ with $V(0)=1$ and $V>0$. Since $\Gamma_j\to\infty$, \eqref{eq_sharp_bound} gives $V\le 1$. Theorem \ref{thm_profile} gives $V=U$.
\end{proof}

\begin{lemma}\label{lem_distance_selection}
Let $r_*>0$, $e\in\mathbb{R}^N$, and let $\Theta>0$ be continuous on $\overline{B_{r_*}(e)}$. With $d(x)=r_*-|x-e|$ there is $z\in B_{r_*}(e)$ such that $d(z)\Theta(z)\ge\frac{r_*}2\sup_{B_{r_*/2}(e)}\Theta$ and $\Theta\le 2\Theta(z)$ on $B_{d(z)/2}(z)$.
\end{lemma}

\begin{proof}
Let $z$ maximize $d\,\Theta$ on $\overline{B_{r_*}(e)}$. Since $d\ge r_*/2$ on $B_{r_*/2}(e)$, the first claim holds, and $z$ is interior. If $|x-z|\le d(z)/2$, then $d(x)\ge d(z)/2$, and $d(x)\Theta(x)\le d(z)\Theta(z)$ gives $\Theta(x)\le 2\Theta(z)$.
\end{proof}

\begin{lemma}[No concentration at a contact scale]\label{lem_no_concentration}
Let $e_j\in\partial\hat{\mathcal{W}}_1$, let $\varepsilon_j>0$, and let $W_j$ be positive weak solutions of
\[
-\Delta_p^H W_j=\varepsilon_jW_j^q
\]
in a neighborhood of $B_{8r_1}(e_j)$, where $r_1>0$ is fixed and $B_{8r_1}(e_j)\subset\hat{\mathcal{W}}_2$. Assume that, for constants $C_0,C_{\rm M}>0$ independent of $j$,
\[
\varepsilon_j^aW_j\le C_0\ \text{ on }B_{8r_1}(e_j),
\qquad
\varpi_j(B_r(e_j))\le C_{\rm M}r^{N-p}\quad(0<r\le r_1),
\qquad d\varpi_j=\varepsilon_jW_j^q\,dx.
\]
Then there exist $r_*\in(0,r_1/2]$ and $C_*>0$, independent of $j$, such that
\[
\varepsilon_j^{1/p}W_j^{1/a}\le C_*
\qquad\text{on }B_{r_*/2}(e_j).
\]
\end{lemma}

\begin{proof}
Set $\Theta_j=\varepsilon_j^{1/p}W_j^{1/a}$. We first record a scale-invariant lower bound for the mass of an intrinsic ball. Let $z$ be a point, put $\varrho=\Theta_j(z)^{-1}$, and assume that $B_{2\varrho}(z)$ lies in the domain and that $W_j\le 2^aW_j(z)$ on $B_{2\varrho}(z)$. Then
\[
\widetilde V(y)=\frac{W_j(z+\varrho y)}{W_j(z)}
\]
satisfies $\widetilde V(0)=1$, $0<\widetilde V\le2^a$ on $B_2$, and $-\Delta_p^H\widetilde V=\widetilde V^q$ there, because $q-p+1=p/a$. Lemma \ref{lem_harnack}(iii) gives $|\nabla\widetilde V|\le C_2$ on $B_1$, with $C_2$ structural. Hence $\widetilde V\ge1/2$ on $B_{\varrho_0}$, where $\varrho_0=\min\{1,(2C_2)^{-1}\}$. Using $1-N/p=-a$ and $q-N/a=-1$, we obtain
\begin{equation}\label{eq_quantum}
\varpi_j\big(B_{\varrho_0\varrho}(z)\big)
=\varepsilon_jW_j(z)^q\varrho^N\int_{B_{\varrho_0}}\widetilde V^q\,dy
\ge \frac{|B_{\varrho_0}|}{2^q}\frac{1}{\varepsilon_j^aW_j(z)}
\ge \frac{|B_{\varrho_0}|}{2^qC_0}=:m_*.
\end{equation}
Choose $r_*\in(0,r_1/2]$ so small that $C_{\rm M}(2r_*)^{N-p}<m_*$. Suppose, contrary to the conclusion, that along a subsequence
$\sup_{B_{r_*/2}(e_j)}\Theta_j\to\infty$. Lemma \ref{lem_distance_selection}, applied with $d_j'(x)=r_*-|x-e_j|$, gives $z_j\in B_{r_*}(e_j)$ such that
\[
d_j'(z_j)\Theta_j(z_j)\to\infty,
\qquad
\Theta_j\le2\Theta_j(z_j)
\quad\text{on }B_{d_j'(z_j)/2}(z_j).
\]
With $\varrho_j=\Theta_j(z_j)^{-1}$, for large $j$ we have $2\varrho_j\le d_j'(z_j)/2$ and therefore $W_j\le2^aW_j(z_j)$ on $B_{2\varrho_j}(z_j)$. This ball lies in $B_{2r_*}(e_j)\subset B_{8r_1}(e_j)$. By \eqref{eq_quantum},
\[
\varpi_j(B_{2r_*}(e_j))\ge m_*,
\]
whereas the assumed Morrey bound gives $\varpi_j(B_{2r_*}(e_j))\le C_{\rm M}(2r_*)^{N-p}<m_*$. This contradiction proves the lemma.
\end{proof}

\begin{proposition}[First contact]\label{prop_firstcontact}
Let $\Gamma_j\to\infty$ and let $v_j$ be positive weak solutions of $-\Delta_p^Hv_j=v_j^q$ in $\hat{\mathcal{W}}_{4\Gamma_j}$ with $v_j\le C_0$ and $v_j\to U$ in $C^1_{\rm loc}(\mathbb{R}^N)$. Then for every $\vartheta>1$ and every sequence $S_j\ge 1$ with $S_j/\Gamma_j\to 0$,
\[
v_j<\vartheta\,U\quad\text{in }\hat{\mathcal{W}}_{S_j}\qquad\text{for all large }j.
\]
\end{proposition}

\begin{proof}
Suppose the contrary. Passing to a subsequence, for every $j$ there is a point of $\overline{\hat{\mathcal{W}}_{S_j}}$ where $v_j\ge\vartheta U$.

\emph{Step 1: first contact and blow-down.} Let $\rho_j=\inf\{r\in(0,S_j]:\max_{\overline{\hat{\mathcal{W}}_r}}v_j/U\ge\vartheta\}$. Since $v_j/U\to 1$ locally uniformly, $\rho_j\to\infty$. By continuity and minimality, $v_j<\vartheta U$ in $\hat{\mathcal{W}}_{\rho_j}$ and $v_j(\rho_je_j)=\vartheta U(\rho_je_j)$ for some $e_j\in\partial\hat{\mathcal{W}}_1$. Let
\[
W_j(x)=\rho_j^\alpha v_j(\rho_jx),\qquad\varepsilon_j=\rho_j^{-p'},\qquad x\in\hat{\mathcal{W}}_{4\Gamma_j/\rho_j}.
\]
By \eqref{eq_exponents}, $-\Delta_p^HW_j=\varepsilon_jW_j^q$ and $\varepsilon_j^aW_j(x)=v_j(\rho_jx)\le C_0$. The domain bookkeeping will be used repeatedly: since $\rho_j\le S_j$ and $\Gamma_j/S_j\to\infty$, one has
\[
\frac{\Gamma_j}{\rho_j}\ge\frac{\Gamma_j}{S_j}\longrightarrow\infty,
\]
so every fixed compact subset of $\mathbb R^N$ is contained in $\hat{\mathcal{W}}_{4\Gamma_j/\rho_j}$ for all large $j$. By Lemma \ref{lem_bubble}(b),
\begin{equation}\label{eq_contact_value}
W_j(e_j)=\vartheta\rho_j^\alpha U(\rho_je_j)\to\vartheta\kappa_U.
\end{equation}
\emph{Step 2: punctured compactness and the pole.} On $\hat{\mathcal{W}}_1\setminus\{0\}$, $W_j<\vartheta\rho_j^\alpha U(\rho_jx)\le\vartheta\kappa_U\mu(x)$ by Lemma \ref{lem_bubble}(b). On compact subsets $\mathcal{K}$ of $\hat{\mathcal{W}}_1\setminus\{0\}$ the potential $\varepsilon_jW_j^{q-p+1}$ is bounded by $\varepsilon_j(\vartheta\kappa_U\max_{\mathcal{K}}\mu)^{p/a}\to 0$. Lemma \ref{lem_harnack}(iii) gives a subsequence with $W_j\to W$ in $C^1_{\rm loc}(\hat{\mathcal{W}}_1\setminus\{0\})$, where $\Delta_p^HW=0$ weakly and $0\le W\le\vartheta\kappa_U\mu$. For $0<r<1$, Lemma \ref{lem_flux} and the change of variables $y=\rho_jx$ give
\[
\int_{\partial\hat{\mathcal{W}}_r}A(\nabla W_j)\cdot\nu\,dS=-\int_{\hat{\mathcal{W}}_r}\varepsilon_jW_j^q\,dx=-\int_{\hat{\mathcal{W}}_{\rho_jr}}v_j^q\,dy\longrightarrow-\int_{\mathbb{R}^N}U^q=-M_U,
\]
by dominated convergence, since $v_j^q\le\vartheta^qU^q$ on $\hat{\mathcal{W}}_{\rho_j}$ and $v_j\to U$ pointwise. By $C^1$ convergence the flux of $W$ through $\partial\hat{\mathcal{W}}_r$ equals $-M_U\neq 0$. Hence $W\not\equiv 0$, and $W>0$ by the strong minimum principle. Theorem \ref{thm_pole} applies with $R=1$, $C_0=\vartheta\kappa_U$ and $\mathcal{F}=-M_U$. By Theorem \ref{thm_pole}(i) and Lemma \ref{lem_bubble}(b), the constant $\kappa$ equals $\kappa_U$, and
\begin{equation}\label{eq_pole_fc}
W-\kappa_U\mu\to\beta\qquad\text{as }x\to 0.
\end{equation}

\emph{Step 3: $\beta=0$.} Let $f_j(t)=\varepsilon_j(t^+)^q$ and $\mathcal{G}_j(t)=\varepsilon_j(t^+)^{p^*}/p^*$. Since $N/p^*=a$, we have $N\mathcal{G}_j(t)-a\,t\,f_j(t)=0$. Lemma \ref{lem_pohozaev} on the Wulff ball $\hat{\mathcal{W}}_r$, $0<r<1$, gives $\int_{\partial\hat{\mathcal{W}}_r}\mathbf{P}[W_j,\mathcal{G}_j]\cdot\nu\,dS=0$. On $\partial\hat{\mathcal{W}}_r$, $W_j\to W$ in $C^1$ and $\mathcal{G}_j(W_j)\to 0$ uniformly, so $\int_{\partial\hat{\mathcal{W}}_r}\mathbf{P}[W,0]\cdot\nu\,dS=0$. Theorem \ref{thm_pole}(iv) gives $a\beta\mathcal{F}=0$, hence $\beta=0$.

\emph{Step 4: $W=\kappa_U\mu$.} Let $\Phi=\kappa_U\mu$ and fix $\delta\in(0,1/2)$. By \eqref{eq_pole_fc} with $\beta=0$, $W/\mu\to\kappa_U$, so there is $r_\delta\in(0,1/2)$, with $r_\delta\to 0$ as $\delta\to 0$, such that $W>(1-\delta/2)\Phi$ on $\partial\hat{\mathcal{W}}_{r_\delta}$. By uniform convergence, $W_j>(1-\delta)\Phi$ on $\partial\hat{\mathcal{W}}_{r_\delta}$ for large $j$. Let $L_j=2\Gamma_j/\rho_j\to\infty$ and
\[
h_{j,\delta}=(1-\delta)\kappa_U\big(\mu-L_j^{-\alpha}\big)\qquad\text{on }\hat{\mathcal{W}}_{L_j}\setminus\overline{\hat{\mathcal{W}}_{r_\delta}}.
\]
This is a $C^1$ weak solution of $\Delta_p^Hh_{j,\delta}=0$ which vanishes on $\partial\hat{\mathcal{W}}_{L_j}$ and lies below $W_j$ on $\partial\hat{\mathcal{W}}_{r_\delta}$. Since $-\Delta_p^HW_j\ge 0$, Lemma \ref{lem_comparison} gives $W_j\ge h_{j,\delta}$. Letting $j\to\infty$ and then $\delta\to 0$ yields $W\ge\Phi$ in $\hat{\mathcal{W}}_1\setminus\{0\}$.

By Lemma \ref{lem_strongcomp} (with $\nabla\Phi\neq 0$), either $W\equiv\Phi$ or $W>\Phi$ in $\hat{\mathcal{W}}_1\setminus\{0\}$. Suppose $W>\Phi$ and let $g_0=\min_{\partial\hat{\mathcal{W}}_{1/2}}(W-\Phi)>0$. For $\delta\in(0,1)$ and small $\varepsilon>0$, the function $(1-\delta)\Phi+g_0/2$ is a solution below $W$ on $\partial\hat{\mathcal{W}}_{1/2}$, and also on $\partial\hat{\mathcal{W}}_\varepsilon$ because there $W-(1-\delta)\Phi-g_0/2\ge\delta\kappa_U\varepsilon^{-\alpha}-g_0/2>0$. Lemma \ref{lem_comparison} on $\hat{\mathcal{W}}_{1/2}\setminus\overline{\hat{\mathcal{W}}_\varepsilon}$, followed by $\varepsilon\to 0$ and $\delta\to 0$, gives $W\ge\Phi+g_0/2$ in $\hat{\mathcal{W}}_{1/2}\setminus\{0\}$. This contradicts \eqref{eq_pole_fc} with $\beta=0$. Hence
\begin{equation}\label{eq_exact_pole}
W=\kappa_U\mu\qquad\text{in }\hat{\mathcal{W}}_1\setminus\{0\}.
\end{equation}

\emph{Step 5: no concentration near the contact point.} Let $d\varpi_j=\varepsilon_jW_j^q\,dx$. Choose $r_1>0$ such that $B_{8r_1}(e)\subset\hat{\mathcal{W}}_2$ for every $e\in\partial\hat{\mathcal{W}}_1$. For large $j$, $\hat{\mathcal{W}}_2\subset\hat{\mathcal{W}}_{4\Gamma_j/\rho_j}$. For $0<r\le r_1$, Lemma \ref{lem_morrey}, applied to $W_j$ on $B_{4r}(e_j)$, and \eqref{eq_contact_value} give
\begin{equation}\label{eq_morrey}
\varpi_j(B_r(e_j))\le C r^{N-p}W_j(e_j)^{p-1}\le C_{\rm M}r^{N-p},
\end{equation}
with $C_{\rm M}$ independent of $j$. Together with $\varepsilon_j^aW_j\le C_0$, Lemma \ref{lem_no_concentration} yields numbers $r_*>0$ and $C_*>0$, independent of $j$, such that
\begin{equation}\label{eq_theta_bound}
\Theta_j:=\varepsilon_j^{1/p}W_j^{1/a}\le C_*
\qquad\text{on }B_{r_*/2}(e_j).
\end{equation}

\emph{Step 6: conclusion.} On $B_{r_*/2}(e_j)$, $W_j$ solves $-\Delta_p^HW_j=\Theta_j^pW_j^{p-1}$, and \eqref{eq_theta_bound} gives a uniform bound for the coefficient $\Theta_j^p$. By Lemma \ref{lem_harnack}(ii) and \eqref{eq_contact_value}, $W_j$ is bounded on $B_{r_*/8}(e_j)$, and Lemma \ref{lem_harnack}(iii) gives uniform $C^{1,\gamma}$ bounds on $B_{r_*/16}(e_j)$. Passing to a subsequence, $e_j\to e\in\partial\hat{\mathcal{W}}_1$ and $W_j(e_j+\cdot)\to\tilde W$ uniformly on $B_{r_*/32}$, with $\tilde W$ continuous. For $y\in B_{r_*/32}$ with $\hat H_0(e+y)<1$, the $C^1$ convergence $W_j\to W$ near $e+y$ and $e_j\to e$ give $\tilde W(y)=W(e+y)=\kappa_U\mu(e+y)$ by \eqref{eq_exact_pole}. Such $y$ accumulate at $0$ (take $y=-te$ with $t>0$ small), so $\tilde W(0)=\kappa_U\mu(e)=\kappa_U$. On the other hand $\tilde W(0)=\lim_jW_j(e_j)=\vartheta\kappa_U$ by \eqref{eq_contact_value}. Since $\vartheta>1$, this is a contradiction.
\end{proof}

\begin{proof}[Proof of Theorem \ref{thm_schoen}]
If $u\equiv0$, the estimate is immediate. Otherwise Lemma \ref{lem_locreg} makes $u$ positive in the connected Wulff ball, so it suffices to prove the estimate for positive solutions. By translation invariance we may take $x_0=0$. If the estimate fails, there are $R_j>0$ and positive weak solutions $u_j$ in $\mathcal{W}_{3R_j}$ satisfying \eqref{eq_contra}. They are continuous by Lemma \ref{lem_locreg}, so the supremum and infimum in Theorem \ref{thm_schoen} are the maximum and minimum over the closed Wulff balls. Let $x_j,t_j,\sigma_j,\delta_j,\Gamma_j,\Xi_j,v_j$ be given by Proposition \ref{prop_selection}. Put
\[
\eta_j=\Xi_j^{-1/(2(N-p))},\qquad\hat\Gamma_j=\eta_j\Gamma_j.
\]
Then $\eta_j\to 0$ and, by Proposition \ref{prop_selection}(i), $\hat\Gamma_j^{N-p}=\Xi_j^{-1/2}\Gamma_j^{N-p}\ge\Xi_j^{1/2}\to\infty$. Since $H_0\le c_H\hat H_0$, the ball $\hat{\mathcal{W}}_{4\Gamma_j/c_H}$ is contained in $\mathcal{W}_{4\Gamma_j}$, where $v_j$ is defined. Proposition \ref{prop_firstcontact}, applied with $\Gamma_j/c_H$ in place of $\Gamma_j$, $\vartheta=2$ and $S_j=2\hat\Gamma_j=o(\Gamma_j)$, gives $v_j\le 2U\le 2\kappa_U\mu$ in $\hat{\mathcal{W}}_{2\hat\Gamma_j}$ for large $j$, hence $v_j\le 2\kappa_U\hat\Gamma_j^{-\alpha}$ on $\partial\hat{\mathcal{W}}_{\hat\Gamma_j}$.

For $\hat H_0(y)=\hat\Gamma_j$ we have $\delta_jH_0(y)\le c_H\delta_j\hat\Gamma_j=c_H\eta_j\sigma_j$, so $H_0(x_j+\delta_jy)\le 2R_j-8\sigma_j+c_H\eta_j\sigma_j<2R_j$ for large $j$. By the definition of $m_j$,
\[
m_j\le t_j\,v_j(y)\le 2\kappa_Ut_j\hat\Gamma_j^{-\alpha}.
\]
Raising to the power $p-1$ and using $\alpha(p-1)=N-p$ and $\Gamma_j^{N-p}=\sigma_j^{N-p}t_j^p$,
\[
m_j^{p-1}\le C\,t_j^{p-1}\eta_j^{-(N-p)}\sigma_j^{-(N-p)}t_j^{-p}.
\]
Multiplying by $t_j\sigma_j^{N-p}$ gives $\Xi_j\le C\eta_j^{-(N-p)}=C\,\Xi_j^{1/2}$, which contradicts $\Xi_j\to\infty$.
\end{proof}

\section{Global integrability and the proof of Theorem \ref{thm_main}}\label{sec_global}

The integral bootstrap of this section adapts the argument of Ciraolo and Gatti \cite{CG25} for the Euclidean $p$-Laplacian. Anisotropy enters only through the structural constants and through the reflected norm in \eqref{eq_v}. In this section $u$ is a positive local weak solution of \eqref{eq_main}. By Lemma \ref{lem_locreg}, $u\in C^{1,\gamma}_{\rm loc}(\mathbb{R}^N)$. For every $R>0$, Lemma \ref{lem_comparison} applied to $u$ and the constant $\min_{\partial\mathcal{W}_R}u$ gives $\min_{\overline{\mathcal{W}_R}}u=\min_{\partial\mathcal{W}_R}u$.

\begin{lemma}\label{lem_decay}
$u$ is bounded, and there is $C\ge 1$ such that $C^{-1}R^{-\alpha}\le\min_{\overline{\mathcal{W}_R}}u\le CR^{-\alpha}$ for all $R\ge 1$.
\end{lemma}

\begin{proof}
Let $m_1=\min_{\partial\hat{\mathcal{W}}_1}u>0$. For $S>1$ the function $\psi_S=m_1(\mu-S^{-\alpha})/(1-S^{-\alpha})$ is a solution of $\Delta_p^H\psi_S=0$ in $\hat{\mathcal{W}}_S\setminus\overline{\hat{\mathcal{W}}_1}$ (Lemma \ref{lem_bubble}(c)) with $\psi_S=m_1$ on $\partial\hat{\mathcal{W}}_1$ and $\psi_S=0$ on $\partial\hat{\mathcal{W}}_S$. Lemma \ref{lem_comparison} gives $u\ge\psi_S$, and letting $S\to\infty$, $u\ge m_1\mu$ outside $\hat{\mathcal{W}}_1$. Since $\hat H_0\le c_HH_0$, we get $u\ge m_1(c_HR)^{-\alpha}$ on $\overline{\mathcal{W}_R}\setminus\hat{\mathcal{W}}_1$, while $u\ge\min_{\overline{\hat{\mathcal{W}}_1}}u>0$ on $\hat{\mathcal{W}}_1$. Hence $\min_{\overline{\mathcal{W}_R}}u\ge m_2R^{-\alpha}$ for $R\ge 1$, with $m_2>0$ independent of $R$.

By Theorem \ref{thm_schoen} with $x_0=0$ and $R\ge 1$,
\[
\max_{\overline{\mathcal{W}_R}}u\le CR^{p-N}\big(m_2(2R)^{-\alpha}\big)^{1-p}\le C',
\]
because $\alpha(p-1)=N-p$. So $u$ is bounded. Also by Theorem \ref{thm_schoen}, $(\min_{\overline{\mathcal{W}_{2R}}}u)^{p-1}\le CR^{p-N}/u(0)$, which gives the upper bound for $R\ge 2$ after replacing $2R$ by $R$. Enlarging $C$, it also holds for $1\le R\le 2$.
\end{proof}

Let $v=u^{-1/a}$. Since $u$ solves the critical equation for $H$, which is the reflection of $\hat H$, Lemma \ref{lem_transformed} applied to the norm $\hat H$ with $V=u$ and $\Omega=\mathbb{R}^N$ gives
\begin{equation}\label{eq_v}
\operatorname{div}\hat A(\nabla v)=\frac{b\,\hat H(\nabla v)^p+c}{v}\qquad\text{in }\mathcal{D}'(\mathbb{R}^N),
\end{equation}
and $v\in C^1(\mathbb{R}^N)$ is positive. Since $u=v^{-a}$, we have $u^{\gamma}=v^{-a\gamma}$, and $v^{-(b+1)}=u^{(b+1)/a}=u^{p^*-1}$.

\begin{proposition}\label{prop_small_exp}
For every $s<b$ there is $C_s$ with $\int_{\mathcal{W}_{2R}}v^{-s}\,dx\le C_sR^{N-p's}$ for all $R\ge 1$.
\end{proposition}

\begin{proof}
If $s\le 0$, Lemma \ref{lem_decay} gives $u\ge C^{-1}(2R)^{-\alpha}$ on $\mathcal{W}_{2R}$, so $v^{-s}=u^{s/a}\le CR^{-\alpha s/a}=CR^{-p's}$ there, and the claim follows since $|\mathcal{W}_{2R}|=CR^N$.

Let $0<s<b$ and $\gamma=s/a$, so $0<\gamma<b/a=N(p-1)/(N-p)$. The function $u_R(x)=R^\alpha u(Rx)$ is a nonnegative weak supersolution in $\mathbb{R}^N$. Let $\rho_0=2M_H$, so that $\mathcal{W}_2\subset B_{\rho_0}$ by Lemma \ref{lem_norm}(a). Lemma \ref{lem_harnack}(i) and Lemma \ref{lem_decay} give
\[
\int_{\mathcal{W}_2}u_R^\gamma\le\int_{B_{2\rho_0}}u_R^\gamma\le C\Big(\inf_{B_{\rho_0}}u_R\Big)^\gamma\le C\Big(R^\alpha\min_{\overline{\mathcal{W}_{2R}}}u\Big)^\gamma\le C.
\]
Changing variables, $\int_{\mathcal{W}_{2R}}u^\gamma=R^{N-\alpha\gamma}\int_{\mathcal{W}_2}u_R^\gamma\le CR^{N-p's}$, because $\alpha/a=p'$.
\end{proof}

\begin{proposition}\label{prop_Lpstar}
$u\in L^{p^*-1}(\mathbb{R}^N)$.
\end{proposition}

\begin{proof}
Fix $\psi\in C^\infty([0,\infty))$ with $\psi=1$ on $[0,1]$, $\psi=0$ on $[2,\infty)$ and $-2\le\psi'\le 0$, and let $\eta_R(x)=\psi(H_0(x)/R)$. Then $\eta_R\in C^1_c(\mathbb{R}^N)$ and, by Lemma \ref{lem_norm}(a), $H(\pm\nabla\eta_R)+\hat H(\pm\nabla\eta_R)\le C/R$ with $C$ structural.

\emph{Step 1.} Let $s<b+1$ and $R\ge 1$. Test \eqref{eq_v} with $v^{1-s}\eta_R^p$. Using $\hat A(\nabla v)\cdot\nabla v=\hat H(\nabla v)^p$ we get
\[
(b+1-s)\int v^{-s}\hat H(\nabla v)^p\eta_R^p+c\int v^{-s}\eta_R^p=-p\int v^{1-s}\eta_R^{p-1}\hat A(\nabla v)\cdot\nabla\eta_R.
\]
By Lemma \ref{lem_norm}(a),(c) for $\hat H$, the right-hand side is at most $p\int v^{1-s}\eta_R^{p-1}\hat H(\nabla v)^{p-1}\hat H(-\nabla\eta_R)$. By Young's inequality it is bounded by $\frac{b+1-s}2\int v^{-s}\hat H(\nabla v)^p\eta_R^p+C\int v^{p-s}\hat H(-\nabla\eta_R)^p$. Since $s-p<b$, Proposition \ref{prop_small_exp} gives
\begin{equation}\label{eq_step1}
\int_{\mathcal{W}_R}v^{-s}\hat H(\nabla v)^p+\int_{\mathcal{W}_R}v^{-s}\le CR^{-p}\int_{\mathcal{W}_{2R}}v^{p-s}\le CR^{N-p'(s-1)},
\end{equation}
where we used $p'p-p=p'$.

\emph{Step 2.} Test \eqref{eq_v} with $v^{-b}\eta_R$. The terms containing $\hat H(\nabla v)^p$ cancel, and we obtain
\[
c\int_{\mathcal{W}_R}v^{-b-1}\le-\int v^{-b}\hat A(\nabla v)\cdot\nabla\eta_R\le\frac CR\int_{\mathcal{W}_{2R}}v^{-b}\hat H(\nabla v)^{p-1}.
\]
Fix $\varepsilon\in(0,1)$. By H\"older's inequality with exponents $p/(p-1)$ and $p$,
\[
\int_{\mathcal{W}_{2R}}v^{-b}\hat H(\nabla v)^{p-1}\le\Big(\int_{\mathcal{W}_{2R}}v^{-(b+1-\varepsilon)}\hat H(\nabla v)^p\Big)^{\frac{p-1}p}\Big(\int_{\mathcal{W}_{2R}}v^{-s_2}\Big)^{\frac1p},
\]
with $s_2=b-(1-\varepsilon)(p-1)<b$. By \eqref{eq_step1} with $s=b+1-\varepsilon$ and radius $2R$, the first factor is at most $(CR^{N-p'(b-\varepsilon)})^{(p-1)/p}=CR^{\varepsilon}$, since $p'b=N$. By Proposition \ref{prop_small_exp}, the second factor is at most $(CR^{N-p's_2})^{1/p}=CR^{1-\varepsilon}$. Hence $\int_{\mathcal{W}_R}v^{-b-1}\le C$ with $C$ independent of $R\ge 1$. Letting $R\to\infty$ and using $v^{-b-1}=u^{p^*-1}$ gives the claim.
\end{proof}

\begin{proposition}\label{prop_energy}
$u\in D^{1,p}(\mathbb{R}^N)=\{u\in L^{p^*}(\mathbb{R}^N):\nabla u\in L^p(\mathbb{R}^N)\}$.
\end{proposition}

\begin{proof}
By Lemma \ref{lem_decay} and Proposition \ref{prop_Lpstar}, $u^{p^*}\le\|u\|_\infty u^{p^*-1}\in L^1(\mathbb{R}^N)$. Test \eqref{eq_main} with $u\eta_R^p$, where $\eta_R$ is as in the proof of Proposition \ref{prop_Lpstar}; this is admissible by the standard local density extension of the weak formulation. As in Step 1 of that proof,
\[
\int\eta_R^pH(\nabla u)^p\le 2\int u^{p^*}+C\int u^pH(-\nabla\eta_R)^p\le 2\int u^{p^*}+\frac{C}{R^p}\Big(\int_{\mathbb{R}^N}u^{p^*}\Big)^{p/p^*}|\mathcal{W}_{2R}|^{p/N},
\]
by H\"older's inequality, since $1-p/p^*=p/N$. The right-hand side is bounded independently of $R$, and letting $R\to\infty$ gives $\nabla u\in L^p(\mathbb{R}^N)$.
\end{proof}

\begin{proof}[Proof of Theorem \ref{thm_main}]
Let $u$ be a local weak solution of \eqref{eq_main}. By Lemma \ref{lem_locreg}, either $u\equiv 0$ or $u>0$. In the second case $u\in D^{1,p}(\mathbb{R}^N)$ by Proposition \ref{prop_energy}. Since $u^{p^*-1}\in L^{(p^*)'}(\mathbb{R}^N)$, the weak formulation extends by density to test functions in $D^{1,p}(\mathbb{R}^N)$, and $u$ is a solution of problem (1.8) of \cite{CFR20} with $\Sigma=\mathbb{R}^N$. We verify explicitly that the anisotropy falls under the hypotheses of \cite[Theorem 1.1]{CFR20}. Indeed,
\[
D^2\!\left(\frac{H^2}{2}\right)=H D^2H+\nabla H\otimes\nabla H,
\]
so \eqref{eq_unif_convex} is precisely the uniform positive-definiteness required there. It also yields the global $C^{1,1}$ regularity of $H^2$ required in \cite{CFR20}: indeed $F=H^2/2$ is $2$-homogeneous, hence $F(\xi)=O(|\xi|^2)$ and $\nabla F(\xi)=O(|\xi|)$, so $F$ extends as a $C^1$ function at the origin with $\nabla F(0)=0$; moreover $\|D^2F\|\le\Lambda$ on $\mathbb R^N\setminus\{0\}$ by \eqref{eq_unif_convex}. Integrating $D^2F$ along line segments (splitting at the origin if a segment passes through it) gives $|\nabla F(\xi)-\nabla F(\eta)|\le\Lambda|\xi-\eta|$. Thus $H^2\in C^{1,1}(\mathbb R^N)$, while the required $C^2$ regularity away from the origin and positive one-homogeneity of $H$ are already part of Hypothesis~(H). Consequently \cite[Theorem 1.1]{CFR20} gives $u(x)=U^H_{\lambda,x_0}(x)$ in the notation of \cite[(1.6)]{CFR20}. The norm $\hat H_0$ appearing there is our $\hat H_0$, so $U^H_{\lambda,x_0}$ has the form $\varsigma(\tau+\hat H_0(x-x_0)^{p'})^{-a}$ with $\tau>0$. The functions $U_{z,\lambda}$ of \eqref{eq_bubbles} have the same form, with $\tau=\lambda^{p'}/K$ taking every positive value. For fixed $\tau$ and $z$, exactly one $\varsigma>0$ gives a solution, because the left-hand side of \eqref{eq_main} is homogeneous of degree $p-1$ in $\varsigma$ and the right-hand side of degree $p^*-1\neq p-1$. Hence $u=U_{z,\lambda}$ for some $z$ and $\lambda$.
\end{proof}

\section*{Declarations}

\smallskip
\noindent\textbf{Conflict of interest.} The author declares that he has no conflict of interest.

\smallskip
\noindent\textbf{Data availability.} Data sharing is not applicable to this article, as no datasets were generated or analyzed during the current study.

\end{document}